\documentclass[12pt]{amsart}
\usepackage{amsmath,times,epsfig,amssymb,amsbsy,amscd,amsfonts,amstext,color,bm}
 \usepackage[margin=1in]{geometry}
\usepackage{enumerate}
\usepackage{placeins}
\usepackage{array}
\usepackage{tabulary}
\usepackage{multirow}
\usepackage{graphicx}
\usepackage{hyperref}
\usepackage{hhline}
\usepackage{tabu}
\usepackage[table]{xcolor}
\usepackage{subcaption}
\usepackage{soul}
\usepackage[color,matrix,arrow]{xypic}
\hypersetup{
  colorlinks   = true, 
  urlcolor     = blue, 
  linkcolor    = blue, 
  citecolor   = red 
}

\theoremstyle{plain}
\newtheorem{theorem}{Theorem}[section]
\newtheorem{corollary}[theorem]{Corollary}
\newtheorem{lemma}[theorem]{Lemma}

\newtheorem{proposition}[theorem]{Proposition}

\newtheorem{notation}[theorem]{Notation}

\makeatletter
    
    \@addtoreset{equation}{section}
  \makeatother

\theoremstyle{definition}
\newtheorem{definition}[theorem]{Definition}
\newtheorem{example}[theorem]{Example}

\theoremstyle{remark}
\newtheorem{remark}[theorem]{Remark}

\newcommand{\A}{\mathcal{A}}
\newcommand{\B}{\mathcal{B}}
\newcommand{\K}{\mathcal{K}}
\newcommand{\U}{\mathcal{U}}

\newcommand{\N}{\mathcal{N}}

\newcommand{\R}{\mathbb{R}}

\newcommand{\scS}{\mathcal{S}}

\newcommand{\Z}{\mathbb{Z}}

\newcommand{\h}{{\rm h}}
\newcommand{\vn}{\noindent}

\newcommand{\D}{{\mathcal{D}}}
\newcommand{\T}{{\mathcal{T}}}

 \newcommand{\lan}{\langle}
\newcommand{\ran}{\rangle}
\newcommand{\Der}{{\rm Der}} 

\newcommand{\codim}{{\rm codim}}

\newcommand{\M}{\mathcal{M}}

\newcommand{\rank}{\operatorname{rank}}

\newcommand{\DP}{{\mathcal{DP}}}

\newcolumntype{K}[1]{>{\centering\arraybackslash}p{#1}}

\allowdisplaybreaks
\begin{document}

\title[On $A_1^2$ restrictions of Weyl arrangements]{On $A_1^2$ restrictions of Weyl arrangements}

\date{\today}

\begin{abstract}
Let $\mathcal{A}$ be a Weyl arrangement in an $\ell$-dimensional Euclidean space. The freeness of restrictions of $\mathcal{A}$ was first settled by a case-by-case method by Orlik and the second author (1993), and later by a uniform argument by Douglass (1999). Prior to this, Orlik and Solomon (1983) had completely determined the exponents of these arrangements by exhaustion. A classical result due to Orlik, Solomon and the second author (1986), asserts that the exponents of any $A_1$ restriction, i.e., the restriction of $\mathcal{A}$ to a hyperplane, are given by $\{m_1,\ldots, m_{\ell-1}\}$, where $\exp(\mathcal{A})=\{m_1,\ldots, m_{\ell}\}$ with $m_1 \le \cdots\le m_{\ell}$. As a next step towards conceptual understanding of the restriction exponents we will investigate the $A_1^2$ restrictions, i.e., the restrictions of $\mathcal{A}$ to the subspaces of type $A_1^2$. In this paper, we give a combinatorial description of the exponents and describe bases for the modules of derivations of the $A_1^2$ restrictions in terms of the classical notion of related roots by Kostant (1955).
 \end{abstract}

\author{Takuro Abe}
\address{Takuro Abe, Institute of Mathematics for Industry, Kyushu University, Fukuoka 819-0395, Japan}
\email{abe@imi.kyushu-u.ac.jp}

\author{Hiroaki Terao}
\address{Hiroaki Terao, Department of Mathematics, Hokkaido University, Kita 10, Nishi 8, Kita-Ku, Sapporo 060-0810, Japan}
\email{hterao00@za3.so-net.ne.jp}

\author{Tan Nhat Tran}
\address{Tan Nhat Tran, Department of Mathematics, Hokkaido University, Kita 10, Nishi 8, Kita-Ku, Sapporo 060-0810, Japan}
\email{trannhattan@math.sci.hokudai.ac.jp}




\date{\today}
\maketitle

\tableofcontents


\section{Introduction}
Assume that $V=\R^\ell$ with the standard inner product $(\cdot,\cdot)$.
Denote by $\Phi$ an irreducible (crystallographic) root system in $V$ and by $\Phi^+$ a positive system of $\Phi$.
Let $\A$ be the Weyl arrangement of $\Phi^+$. 
Denote by $L(\A)$ the intersection poset of $\A$.
For each $X\in L(\A)$, we write ${\A}^{X}$ for the restriction of ${\A}$ to $X$. 
Set $L_p(\A):=\{X \in L(\A) \mid \codim(X)=p\}$ for  $0 \le p \le \ell$.
Let $W$ be the Weyl group of $\Phi$ and let $m_1, \ldots, m_\ell$ with $m_1 \le \cdots \le m_\ell$ be the exponents of $W$.

\begin{notation}\label{not:associated}
If $X \in L_p(\A)$, then $\Phi_X := \Phi \cap X^{\perp}$ is a root system of rank $p$. 
A positive system of $\Phi_X$ is taken to be $\Phi^+_X:=\Phi^+\cap\Phi_X$.
Let $\Delta_X$ be the base of  $\Phi_X$ associated to  $\Phi^+_X$. 
\end{notation}

\begin{definition}\label{def:s}
A subspace $X\in L(\A)$ is said to be \emph{of type $T$} (or \emph{$T$} for short) if $\Phi_X$ is a root system of type $T$. 
In this case, the restriction $\A^{X}$ is said to be \emph{of type $T$} (or \emph{$T$}).
\end{definition}

Weyl arrangements are important examples of \emph{free} arrangements. 
In other words, the module $D(\A)$ of $\A$-derivations is a free module. 
Furthermore, the exponents of $\A$ are the same as the exponents of $W$, i.e., $\exp(\mathcal{A})=\{m_1,\ldots, m_{\ell}\}$ (e.g., \cite{S93}).
It is shown by Orlik and Solomon  \cite{OS83}, using the classification of finite reflection groups, that the characteristic polynomial of the restriction ${\A}^{X}$ ($X \in L(\A)$) of an arbitrary Weyl arrangement $\A$ is fully factored. 
Orlik and the second author  \cite{OT93} proved a stronger statement that $D(\A^X)$ is free by a case-by-case study. 
Soon afterwards, Douglass \cite{D99} gave a uniform proof for the freeness using the representation theory of Lie groups. 

We are interested in  studying  the exponents of $\A^X$  and bases for $D(\A^X)$.
For a general $X$, it is not an easy task. 
The only known general result due to Orlik, Solomon and the second author \cite{OST86},
asserts that for each $X\in\A$, i.e., $X$ is of type $A_1$, $\exp(\A^X)=\{m_1,\ldots, m_{\ell-1}\}$ and a basis for $D(\A^X)$ consists of the restrictions to $X$ of the corresponding basic derivations. 
In this paper, we consider a ``next" general class of restrictions, that is when $X$ is of type $A_1^2$.
We prove that $\exp(\A^X)$ is obtained from $\exp(\A)$ by removing either the two largest exponents, or the largest and the middle exponents, depending upon a combinatorial condition on $X$.
Furthermore, similar to the result of  \cite{OST86}, our method produces an explicit basis for  $D(\A^X)$ in each case.
The main combinatorial ingredient in our description is the following concept defined by Kostant:\begin{definition}[\cite{K55}]\label{def:RO}
Two non-proportional roots $\beta_1$, $\beta_2$ are said to be \textit{related} if 
\begin{enumerate}[(a)]
\item $(\beta_1,\beta_2)=0$,
\item for any $\gamma\in\Phi \setminus \{\pm\beta_1,\pm\beta_2\}$, $(\gamma,\beta_i)=0$
implies $(\gamma,\beta_{3-i})=0$ for all $i \in\{1,2\}$.
\end{enumerate}
In this case, we call the set $\{\beta_1,\beta_2\}$ \textit{relatedly orthogonal} (RO), and the subspace $X=H_{\beta_1} \cap H_{\beta_2}$ is said to be RO.
\end{definition}

\begin{remark}\label{rem:terminology}
The relatedly orthogonal sets presumably first appeared in \cite{K55}, wherein Kostant required $\beta_1,\beta_2$ to have the same length, and allowed a root is related to itself and its negative. 
Green called the relatedly orthogonal sets \textit{strongly orthogonal} and defined the strong orthogonality in a more general setting \cite[Definition 4.4.1]{G13}. 
It should be noted that the notion of strongly orthogonal sets is probably more well-known with the definition that neither sum nor difference of the two roots is a root. 
For every $\beta \in \Phi$, set $\beta^\perp:=\{\alpha \in \Phi \mid (\alpha,\beta)=0\}$. Condition (b) in Definition \ref{def:RO} can be written symbolically as (b') $\beta_1^{\perp}\setminus \{\pm \beta_2\}=\beta_2^{\perp} \setminus \{\pm \beta_1\}$.
\end{remark}

Let $\h$ be the Coxeter number of $W$.
For $\phi \in D(\A)$, let $\phi^X$ be the restriction of $\phi$ to $X$.
We now formulate our main results.
 \begin{theorem}\label{thm:combine}
 Assume that $\ell \ge 3$.
If $X\in L(\A)$ is of type $A_1^2$, then $\A^X$  is free with 
\begin{equation*}\label{eq:exp-part}
\exp(\A^X)=
\begin{cases}
\exp(\A)\setminus \{ \h/2,m_\ell\} \mbox{ if $X$ is RO}, \\
\exp(\A)\setminus \{ m_{\ell-1},m_\ell\} \mbox{ otherwise}.
\end{cases}
\end{equation*}
 
\end{theorem}

\begin{theorem}\label{thm:basis-main}
 Assume that $\ell \ge 3$.
\begin{enumerate}[(i)]
\item Suppose that $X$ is  $A_1^2$  and not RO. 
Let $\{\varphi_{1}, \ldots, \varphi_\ell\}$ be a basis for $D(\A)$ with $\deg {\varphi}_j=m_j$ ($1 \le j \le \ell$). 
Then $\{ {\varphi}^X_1, \ldots, {\varphi}^X_{\ell-2}\}$ is a basis for $D(\A^X)$.
\item Suppose that $X$ is both $A_1^2$ and RO. 
Then $\Phi$ must be of type $D_\ell$ with $\ell \ge 3$. 
Furthermore, a basis for $D(\A^X)$ is given by $\{ \tau_1^X, \ldots, \tau_{\ell-2}^X\}$, where
$$
\tau_i := \sum_{k=1}^\ell x_k^{2i-1}\partial/\partial x_k\,\, (1 \le i \le \ell-2).
$$
\end{enumerate}
 \end{theorem}

Theorem \ref{thm:combine} gave a little extra information: when $X$ is RO, half of the Coxeter number is an exponent of $W$ (hence it lies in the ``middle" of the exponent sequence). 
We emphasize that given \cite{OT93}, Theorem \ref{thm:combine} can be verified by looking at the numerical results in \cite{OS83} (the case of type $D$ is rather non-trivial). 
It is interesting to search for a proof, free of case-by-case considerations.
In this paper, we provide a conceptual proof with a minimal use of classifications of root systems: Theorem \ref{thm:combine} holds true trivially for $\ell=3$, and we use up to the classification of rank-$4$ root systems (for the proofs of some supporting arguments). 
As far as we are aware, a formula of the exponents of a restricted Weyl arrangement given under the RO condition is new.
Our proof shows how the RO concept arrives at the exponent description, hoping that it will reveal a new direction for future research of the exponents through the combinatorial properties of the root system.

Theorem \ref{thm:basis-main} seems to be, however, less straightforward even if one relies on the classification. 
In comparison with the classical result of  \cite{OST86}, Theorem \ref{thm:basis-main}(i) gives a bit more flexible construction of basis for $D(\A^X)$. 
Namely, a wanted basis is obtained by taking the restriction to $X$ of \emph{any} basis for $D(\A)$, without the need of basic derivations. 
Nevertheless, the remaining part of the basis construction (Theorem \ref{thm:basis-main}(ii)) can not avoid the classification. 
It may happen that there are more than one derivations in a basis for $D(\A)$ having the same degree. 
Hence some additional computation is required to examine which derivation vanishes after taking the restriction to $X$ (see Remark \ref{rem:once} and Example \ref{ex:basis-RO}).

Beyond the $A_1^2$ restrictions, driving conceptual understanding on the restrictions in higher codimensions, or of irreducible types is much harder (see Remarks \ref{rem:A2,B2} and \ref{rem:A1k}). 
The numerical results  \cite{OS83} say that we have a similar formula for the exponents of any $A_1^k$ restriction ($k\ge 2$) in terms of the strongly orthogonal sets of  \cite{G13}. 
The details are left for future research.

The remainder of this paper is organized as follows. 
In \S\ref{sec:pre}, we first review preliminary results on free arrangements and their exponents, and a recent result (Combinatorial Deletion) relating the freeness to combinatorics of arrangements (Theorem \ref{thm:Abe}). 
We also prove an important result in the paper, a construction of a basis for $D(\A^H)$ ($H\in\A$) from a basis for $D(\A)$ when $\A$ and $\A\setminus\{H\}$ are both free (Theorem \ref{thm:basis-derived}).
We then review the background information on root systems, Weyl groups, and Weyl arrangements. 
Based on the Combinatorial Deletion Theorem, we provide a slightly different proof for the result of  \cite{OST86} (Remark \ref{rem:alternative}).
In \S\ref{sec:cardinality}, we give an evaluation for the cardinality of every $A_1^2$ restriction (Proposition \ref{prop:X=A_1^2}). 
This evaluation can be expressed in terms of the local and global second smallest exponents of the Weyl groups (Remark \ref{rem:K0}). 
In \S\ref{sec:main}, 
we first provide a proof for the freeness part of Theorem \ref{thm:combine}, i.e., the freeness of $A_1^2$ restrictions (Theorem \ref{thm:main}). 
The proof is different from (and more direct than) the proofs presented in \cite{OT93} and \cite{D99}. 
We then complete the proof of Theorem \ref{thm:combine} by providing a proof for the exponent part (Theorem  \ref{thm:card}). 
The proof contains two halves which we present the proof for each half in Theorems \ref{thm:half1} and \ref{thm:half2}. 
We close the section by giving two results about local-global inequalities on the second smallest exponents and the largest coefficients of the highest roots (Corollary \ref{cor:Local-global}), and the Coxeter number of any irreducible component of the subsystem orthogonal to the highest root in the simply-laced cases (Corollary \ref{cor:Omega-cx}). 
In \S\ref{sec:Weyl-freeness part}, we present the proof of Theorem \ref{thm:basis-main} (Theorem \ref{thm:basis-non-RO} and Example \ref{ex:basis-RO}). 
In \S\ref{sec:app}, we give an alternative and bijective proof of Theorem \ref{thm:crucial}, one of the key ingredients in the proof of the exponent part of Theorem \ref{thm:combine}.
\section{Preliminaries}\label{sec:pre}
\subsection{Free arrangements and their exponents}
For basic concepts and results of free arrangements, we refer the reader to \cite{OT92}.

Let $\Bbb K$ be a field and let $V := \Bbb K^\ell$.
A \textit{hyperplane} in $V$ is a subspace of codimension $1$ of $V$.
An \textit{arrangement} is a finite set of hyperplanes in $V$.
We choose a basis $\{x_1,\ldots, x_\ell\}$ for $V^*$ and let $S:= \Bbb K[x_1,\ldots, x_\ell]$.
Fix an arrangement $\A$ in $V$. 
The \textit{defining polynomial} $Q(\A)$ of $\A$ is defined by
$$Q(\A):=\prod_{H \in \A} \alpha_H \in S,$$
where $ \alpha_H=a_1x_1+\cdots+a_\ell x_\ell \in V^*\setminus\{0\}$,  $a_i \in \Bbb K$ and $H = \ker \alpha_H$.
The \textit{number} of hyperplanes in $\A$ is denoted by $|\A|$.
It is easy to see that $Q(\A)$ is a homogeneous polynomial in $S$ and $\deg Q(\A)=|\A|$.

The \emph{intersection poset} of $\A$, denoted by $L(\A)$, is defined to be
$$L(\A):=\{ \cap_{H \in \B} H \mid \B \subseteq \A \},$$
where the partial order is given by reverse inclusion. We agree that $V \in L(\A)$ is the unique minimal element.
For each $X \in L(\A)$, we define the \emph{localization} of $\A$ on $X$  by 
$${\A}_X := \{ K \in {\A} \mid X \subseteq K\},$$
and define the \emph{restriction} ${\A}^{X}$ of ${\A}$ to $X$ by 
$${\A}^{X}:= \{ K \cap X \mid K \in{\A }\setminus {\A}_X\}.$$
The \textit{M\"{o}bius function} $\mu : L(\A) \to \mathbb Z$ is formulated by
$$\mu(V) :=1, \quad \mu(X) := - \sum_{X \subsetneq Y \subseteq V}\mu(Y).$$
The \textit{characteristic polynomial} $\chi(\A, t)$ of $\A$ is defined by
$$ \chi (\A, t):= \sum_{X \in L(\A)} \mu(X) t^{\dim X}.$$

A \textit{derivation} of $S$ over $ \Bbb K$ is a linear map $\phi:S\to S$ such that for all $f,g \in S$,
$\phi(fg)=f\phi(g)+g\phi(f).$
Let ${\rm{Der}}(S)$ denote the \emph{set} of derivations of $S$ over $ \Bbb K$.
Then ${\rm{Der}}(S)$ is a free $S$-module with a basis $\{\partial_1,\ldots , \partial_\ell\}$, where $\partial_i:=\partial/\partial x_i$ for  $1 \le i \le \ell$.
Define an $S$-submodule of ${\rm{Der}}(S)$, called the \textit{module of $\A$-derivations}, by
$$D(\A):= \{ \phi \in \Der(S) \mid \phi(Q) \in QS\}.$$
A non-zero element $\phi= f_1\partial_1+\cdots+f_\ell\partial_\ell\in  \Der(S)$ is \textit{homogeneous of degree $b$} if each non-zero polynomial $f_i\in S$ for $1 \le i \le \ell$ is homogeneous of degree $b$. 
We then write $\deg\phi=b$.
The arrangement $\A$ is called \textit{free} if $D(\A)$ is a free  $S$-module.
If $\A$ is free, then $D(\A)$ admits a basis $\{\phi_1, \ldots , \phi_\ell\}$ consisting of homogeneous derivations \cite[Proposition 4.18]{OT92}. 
Such a basis is called a homogeneous basis. 
Although homogeneous basis need not be unique, the degrees of elements of a basis are unique (with multiplicity but neglecting the order) depending only on $\A$ \cite[Proposition A.24]{OT92}.
In this case, we call $\deg\phi_1,  \ldots , \deg\phi_\ell$ the \textit{exponents} of $\A$, store them in a \textit{multiset} denoted by $\exp (\A)$ and write
$$\exp (\A)= \{\deg\phi_1,  \ldots , \deg\phi_\ell\}.$$
Interestingly, when an arrangement is free, the exponents turn out to be the roots of the characteristic polynomial due to  the second author.
\begin{theorem}[Factorization]\label{thm:Factorization}
If $\A$ is free with $\exp(\A) = \{d_1, \ldots, d_\ell\}$, then 
$$\chi (\A, t)= \prod_{i=1}^\ell (t-d_i).$$
\end{theorem}
\begin{proof} 
See \cite{T81} or \cite[Theorem 4.137]{OT92}.
\end{proof}

For $\phi_1, \ldots , \phi_\ell \in D(\A)$, we define the $(\ell \times\ell)$-matrix $M(\phi_1, \ldots , \phi_\ell)$ as the matrix with $(i,j)$th entry $\phi_j(x_i)$. 
In general, it is difficult to determine whether a given arrangement is free or not. 
However, using the following criterion, we can verify that a candidate for a basis is actually a basis.
\begin{theorem}[Saito's criterion]\label{thm:criterion}
Let $\phi_1, \ldots , \phi_\ell \in D(\A)$. Then $\{\phi_1, \ldots , \phi_\ell\}$ forms a basis for $D(\A)$ if and only if
$$\det M(\phi_1, \ldots , \phi_\ell) = cQ(\A)\,\,(c\ne0).$$
In particular, if $\phi_1, \ldots , \phi_\ell$ are all homogeneous, then $\{\phi_1, \ldots , \phi_\ell\}$ forms a basis for $D(\A)$ if and only if the following two conditions are satisfied:
\begin{enumerate}[(i)]
\item $\phi_1, \ldots , \phi_\ell$ are independent over $S$, 
\item  $\sum_{i=1}^\ell \deg\phi_i = |\A|$.
\end{enumerate}
\end{theorem}
\begin{proof} 
See \cite[Theorems 4.19 and 4.23]{OT92}.
\end{proof}

\vn
In addition to the Saito's criterion, we have a way to check if a set of derivations is part of a homogeneous basis, and sometimes a full basis. 
First, the notation $\{d_1, \ldots, d_\ell\}_{\le}$􏳨 indicates $d_1\le \cdots\le d_\ell$.
\begin{theorem}\label{thm:part}
Let $\A$ be a free arrangement with  $\exp(\A) =\{d_1, \ldots, d_\ell\}_{\le}$. 
If $\phi_1, \ldots , \phi_k \in D(\A)$ satisfy for $1 \le i\le k$,
\begin{enumerate}[(i)]
\item $\deg\phi_i =d_i$, 
\item  $\phi_i \notin S\phi_1 +\cdots+S\phi_{i-1}$,
\end{enumerate}
then $\phi_1, \ldots , \phi_k$ may be extended to a basis for $D(\A)$.
\end{theorem}
\begin{proof} 
See \cite[Theorem 4.42]{OT92}.
\end{proof}

 \begin{definition}\label{def:X-bar}
 For $X\in L(\A)$, let $I=I(X):=\sum_{H\in\A_X}\alpha_HS$ and $\overline{S}:=S/I$.
For $\phi \in D(\A)$, define $\phi^X \in \Der(\overline{S})$ by $\phi^X(f+I)=\phi(f)+I$. 
We call $\phi^X$ the restriction of $\phi$ to $X$.
 \end{definition}
 \begin{proposition}\label{prop:general-X}
If $\phi \in D(\A)$, then $\phi^X\in D(\A^X)$. 
If  $\phi^X \ne 0$, then $\deg\phi^X=\deg\phi$.
  \end{proposition} 
\begin{proof}
See   \cite[Lemma 2.12]{OST86}.
\end{proof}

Fix $H \in \A$, denote $\A':=\A\setminus \{H\}$ and $\A'':=\A^H$. 
We call $(\A, \A', \A'')$ the triple with respect to the hyperplane $H \in\A$. 
 
\begin{proposition}\label{prop:exact}
Define $h : D(\A') \to D(\A)$ by $h(\phi) = \alpha_H\phi$ and $q : D(\A) \to
D(\A'')$ by $q(\phi) ={\phi}^H$. The sequence
$$0 \rightarrow D(\A') \stackrel{h}{\rightarrow}   D(\A)  \stackrel{q}{\rightarrow} D(\A'')$$ 
is exact.
 \end{proposition} 
\begin{proof}
See  \cite[Proposition 4.45]{OT92}.
\end{proof}

Then the freeness of any two of the triple, under a certain condition on their exponents, implies the freeness of the third.
\begin{theorem}[Addition-Deletion]\label{thm:AD}
Let $\A$ be a non-empty arrangement and let $H \in \A$. Then two of the following imply the third:
\begin{enumerate}[(1)]
\item $\A$ is free with $\exp(\A) = \{d_1, \ldots, d_{\ell-1}, d_\ell\}$.
\item  $\A'$ is free with $\exp(\A')=\{d_1, \ldots, d_{\ell-1}, d_\ell-1\}$. 
\item $\A''$ is free with $\exp(\A'') = \{d_1, \ldots, d_{\ell-1}\}$.
\end{enumerate}
Moreover, all the three hold true if $\A$ and $\A'$ are both free.
\end{theorem}
\begin{proof} 
See \cite{T80} or  \cite[Theorems 4.46 and 4.51]{OT92}.
\end{proof}

The Addition Theorem and Deletion Theorem above are rather ``algebraic" as they rely on the conditions concerning the exponents. 
Recently, the first author has found ``combinatorial" versions for these theorems \cite{Abe18}, \cite{Abe18AD}, \cite{Abe19}. 
In this paper, we focus on the Deletion theorem.
 \begin{theorem}[Combinatorial Deletion]\label{thm:Abe}
 Let $\A$ be a free arrangement and $H \in \A$. 
Then $\A'$ is free if and only if $|\A_X| - |\A^H_X|$ is a root of $\chi(\A_X, t)$ for all $X \in L(\A^H)$. 
\end{theorem}
\begin{proof} 
See \cite[Theorem 8.2]{Abe18}.
\end{proof}

When $\A$ and $\A'$ are both free, one may construct a basis for $D(\A'')$ from a basis for $D(\A)$. 
Although the following theorem is probably well-known among experts, we give a detailed proof for the sake of completeness.
\begin{theorem}\label{thm:basis-derived}
Let $\A$ be a non-empty free arrangement and $\exp(\A)= \{d_1, \ldots, d_\ell\}_{\le}$. 
 Assume further that $\A'$ is also free.
 Let $\{\varphi_{1}, \ldots, \varphi_\ell\}$ be a basis for $D(\A)$ with $\deg {\varphi}_j=d_j$ for $1 \le j \le \ell$. 
Then there exists some $p$ with $1 \le p \le \ell$ such that $\{ {\varphi}^H_1, \ldots, {\varphi}^H_\ell\}  \setminus \{{\varphi}^H_{p}\}$ forms a basis for $D(\A'')$.
\end{theorem}
 \begin{proof}
By Theorem \ref{thm:AD},  $\A''$ is also free and we may write $\exp(\A'')=\exp(\A)\setminus\{d_k\}$ for some $k$ with $1 \le k \le \ell$.
Denote ${\tau}_i := {\varphi}_i$ ($1 \le i \le k-1$), and ${\tau}_j :=  {\varphi}_{j+1}$  ($k \le j \le \ell-1$).

If ${\tau}^H_i \notin  \overline{S}{\tau}^H_1 + \cdots +  \overline{S}{\tau}^H_{i-1}$ for all $1 \le i \le \ell-1$, then by Theorem \ref{thm:part}, $\{{\tau}^H_1, \ldots, {\tau}^H_{\ell-1}\}=\{ {\varphi}^H_1, \ldots, {\varphi}^H_\ell\}  \setminus \{{\varphi}^H_{k}\}$ forms a basis for $D(\A'')$.

If not, there exists some $p$, $1􏳨 \le p \le 􏳨\ell-1$, such that 
${\tau}^H_p \in  \overline{S}{\tau}^H_1 + \cdots +  \overline{S}{\tau}^H_{p-1}$. 
By Proposition \ref{prop:exact},
$${\tau}_p = f_1{\tau}_1+\cdots+f_{p-1}{\tau}_{p-1}+\alpha_H\tau,$$
where $f_i \in S$ ($1 \le i \le p-1$), $\tau \in D(\A')$, $\deg \tau=d_p-1$. 
By Theorem \ref{thm:criterion} (Saito's criterion), $\{{\tau}_{1},\ldots, {\tau}_{p-1}, \tau, {\tau}_{p+1}, \ldots, {\tau}_{\ell-1}, \varphi_k\}$ is a basis for $D(\A')$. 
By Theorem \ref{thm:AD}, $\exp(\A'') =\exp(\A)\setminus \{d_p\}$.
It means that $d_k=d_p$. 
Note that if $d_k$ appears only once in $\exp(\A)$, then we obtain a contradiction here and the proof is completed. 
Now suppose $d_k$ appears at least twice.
Again by Theorem \ref{thm:criterion}, $\{{\tau}_{1},\ldots, {\tau}_{p-1},\alpha_H \tau, {\tau}_{p+1}, \ldots, {\tau}_{\ell-1}, \varphi_k\}$ is a basis for $D(\A)$.
Since the map $q: D(\A) \to D(\A'')$ is surjective \cite[Proposition 4.57]{OT92}, $D(\A'')$ is generated by $\{ {\tau}^H_{1},\ldots,  {\tau}^H_{p-1},  {\tau}^H_{p+1}, \ldots,  {\tau}^H_{\ell-1}, {\varphi}^H_k\}$. 
This set is  the same as $\{ {\varphi}^H_1, \ldots, {\varphi}^H_\ell\}  \setminus \{{\varphi}^H_{p}\}$ which indeed forms a basis for $D(\A'')$  by \cite[Proposition A.3]{OT92}. 
It completes the proof.
\end{proof}

\begin{remark}\label{rem:once}
If $d_k$ appears only once in $\exp(\A)$, then Theorem \ref{thm:basis-derived} gives an explicit basis for $D(\A'')$. 
However, if $d_k$ appears at least twice, Theorem \ref{thm:basis-derived} may not be sufficient to derive an explicit basis for $D(\A'')$. 
It requires some additional computation to examine which derivation vanishes after taking the restriction to $H$. 
This observation will be useful to construct an explicit basis for $D(\A^X)$ when $X$ is $A_1^2$ and RO (Example \ref{ex:basis-RO}).
\end{remark}


\subsection{Root systems, Weyl groups and Weyl arrangements}
Our standard reference for root systems and their Weyl groups is \cite{B68}.

Let $V := \R^\ell$ with the standard inner product $(\cdot,\cdot)$.
Let $\Phi$ be an irreducible (crystallographic) root system spanning $V$. 
The \emph{rank} of $\Phi$, denoted by $\rm{rank}(\Phi)$, is defined to be dim$(V)$. 
We fix a positive system $\Phi^+$ of $\Phi$.
We write $\Delta:=\{\alpha_1, \ldots ,\alpha_\ell\}$ for the simple system (base) of $\Phi$ associated to  $\Phi^+$. 
For $\alpha = \sum_{i=1}^\ell d_i \alpha_i\in \Phi^+$, the \textit{height} of $\alpha$ is defined by $ {\rm ht}(\alpha) :=\sum_{i=1}^\ell  d_i$. 
There are four classical types: $A_\ell$ ($\ell \ge 1$), $B_\ell$ ($\ell \ge 2$), $C_\ell$ ($\ell \ge 3$), $D_\ell$ ($\ell \ge 4$) and five exceptional types: $E_6$, $E_7$, $E_8$, $F_4$, $G_2$.
We write $\Phi=T$ if the root system $\Phi$ is of type $T$, otherwise we write $\Phi\ne T$.

A \emph{reflection} in $V$ with respect to a vector $\alpha \in V\setminus\{0\}$ is a mapping $s_{\alpha}: V \to V$ defined by
$ s_\alpha (x) := x -2\frac{( x,\alpha )}{( \alpha , \alpha)} \alpha.$ 
The \emph{Weyl group} $W:=W(\Phi)$ of $\Phi$ is a group generated by the set $\{s_{\alpha}\mid \alpha \in \Phi\}$.
 An element of the form $c=s_{\alpha_1}\dots s_{\alpha_\ell}\in W$ is called a \emph{Coxeter element}. 
 Since all Coxeter elements are conjugate (e.g., \cite[Chapter V, $\S$6.1, Proposition 1]{B68}), they have the same order, characteristic polynomial and eigenvalues. 
The order ${\rm h}:={\rm h}(W)$ of Coxeter elements is called the \emph{Coxeter number} of $W$.
For a fixed Coxeter element $c\in W$, if its eigenvalues are of the form $\exp (2\pi\sqrt{-1}m_1/{\rm h}),\ldots, \exp (2\pi\sqrt{-1}m_\ell/{\rm h})$ with $0< m_1 \le \cdots\le m_\ell<{\rm h}$, then the integers  $m_1,\ldots, m_\ell$ are called the \emph{exponents} of $W$ (or of $\Phi$).
\begin{theorem}\label{exponents} 
For any irreducible root system $\Phi$ of rank $\ell$,
\begin{enumerate}[(i)]
\item    $m_j + m_{\ell+1-j}={\rm h}$ for $1 \le j \le \ell$,
\item  $m_1+m_2+\cdots+m_\ell =\ell{\rm h}/2$,
\item   $1=m_1 < m_2 \le \cdots\le m_{\ell-1} <m_\ell={\rm h}-1$,
\item  ${\rm h}= 2\left|\Phi^+ \right|/\ell,$
\item ${\rm h}={\rm ht}(\theta)+1$, where $\theta$ is the highest root of $\Phi$.
\item ${\rm h} = 2 \sum_{\mu \in \Phi^+}( \widehat{\gamma}, \widehat \mu )^2,$ for arbitrary $\gamma\in \Phi^+$. Here $\widehat{x} :=x/(x,x)$.
\end{enumerate}
\end{theorem} 
\begin{proof} 
See, e.g., \cite[Chapter V, $\S$6.2 and Chapter VI, $\S$1.11]{B68}.
\end{proof}

Let $\Theta^{(r)}\subseteq \Phi^+$ be the set consisting of positive roots of height $r$, i.e., $\Theta^{(r)}=\{\alpha \in \Phi^+ \mid {\rm ht}(\alpha)=r\}$. 
The \textit{height distribution} of $\Phi^+$ is defined as a multiset  of positive integers:
$$\{t_1, \ldots , t_r, \ldots , t_{{\rm h}-1}\},$$ where 
$t_r := \left|\Theta^{(r)}\right|$.
The  \textit{dual partition} $\DP(\Phi^+)$ of the height distribution of $\Phi^+$ is given by a multiset of non-negative integers:
$$\DP(\Phi^+) := \{(0)^{\ell-t_1},(1)^{t_1-t_2},\ldots ,({\rm h}-2)^{t_{{\rm h}-2}-t_{{\rm h}-3}},({\rm h}-1)^{t_{{\rm h}-1}}\},$$ 
where notation $(a)^b$ means the integer $a$ appears exactly $b$ times.
 \begin{theorem}\label{thm:dual}
The exponents of $W$ are given by $\DP(\Phi^+)$.
\end{theorem}
\begin{proof} 
See, e.g., \cite{R59},  \cite{K59}, \cite{M72}, \cite{ABCHT16}.
\end{proof}

Let $S^W$ denote the \emph{ring of $W$-invariant polynomials}. 
Let $\mathcal{F} = \{f_1, \ldots, f_\ell\}$ be a set of \emph{basic invariants} with $\deg f_1 \le \cdots\le \deg f_\ell$. 
Then $S^W=\R[f_1, \ldots, f_\ell]$ and $m_i=\deg f_i -1$ ($1 \le i \le \ell$).
Let $\D=\{\theta_{f_1}, \ldots, \theta_{f_{\ell}}\}$ be the set of \emph{basic derivations} associated to $\mathcal{F} $ (see, e.g., \cite[Definition 2.4]{OST86} and \cite[Definition 6.50]{OT92}).
The \emph{Weyl arrangement} of $\Phi^+$ is defined by 
$$\A=\A(\Phi^+):= \{(\Bbb R\alpha)^\perp \mid \alpha\in \Phi^+\}.$$
\begin{theorem} \label{thm:Saito}
$\A$ is free with $\exp(\A) = \{m_1, \ldots, m_\ell\}$ and $\{\theta_{f_1}, \ldots, \theta_{f_{\ell}}\}$ is a basis for $D(\A)$. 
\end{theorem}
\begin{proof} 
See, e.g., \cite{S93} and \cite[Theorem 6.60]{OT92}.
\end{proof}

Recall from Proposition \ref{prop:exact} the map $q : D(\A) \to D(\A'')$ defined by $q(\phi) ={\phi}^H$.
\begin{theorem}\label{thm:OST}
If $H\in\A$, then $\A^H$ is free with $\exp(\A^H)=\{m_1,\ldots, m_{\ell-1}\}$. 
Furthermore, $\{ {\theta}^H_{f_1}, \ldots, {\theta}^H_{f_{\ell-1}}\}$ is a basis for $D(\A^H)$. 
\end{theorem}
\begin{proof} 
See \cite[Theorem 1.12]{OST86}.
\end{proof}

\begin{remark}\label{rem:alternative}
Theorem \ref{thm:OST} can be proved in a slightly different way. By  \cite[Theorem 3.7]{OST86}, $|\A|-|\A^H|=m_\ell$. 
Thus $|\A|-|\A^H|$ is a root of $\chi (\A, t)$ by Theorems \ref{thm:Saito} and \ref{thm:Factorization}. 
By Theorem \ref{thm:Abe}, $\A'$ is free (see also Theorem \ref{thm:main} for a similar and more detailed explanation). 
Thus $\A^H$ is free with $\exp(\A^H)=\{m_1,\ldots, m_{\ell-1}\}$ which follows from Theorem \ref{thm:AD}. 
A basis for $D(\A^H)$ can be constructed a bit more flexibly, without the need to introduce the basic derivations. 
Namely, let $\{\varphi_{1}, \ldots, \varphi_\ell\}$ be \emph{any} basis for $D(\A)$ with $\deg {\varphi}_j=m_j$ for $1 \le j \le \ell$. 
Note that $m_{\ell}$ appears exactly once in $\exp(\A)$.
Then by Theorem \ref{thm:basis-derived}, $\{ {\varphi}^H_1, \ldots, {\varphi}^H_{\ell-1}\}$ is a basis for $D(\A'')$.

\end{remark}

A  subset $\Gamma \subseteq \Phi$ is called a \emph{(root) subsystem} if it is a root system in $\mbox{span}_\R(\Gamma) \subseteq V$.
For any $J \subseteq \Delta$, set $\Phi(J):=\Phi \cap \mbox{span}(J)$.
Let $W(J)$ be the group generated by $\{s_{\delta}\mid \delta \in J\}$. 
By \cite[Proposition 2.5.1]{Ca72}, $\Phi(J)$ is a subsystem of $\Phi$, $J$ is a base of $\Phi(J)$ and the Weyl group of $\Phi(J)$ is $W(J)$.
For any subset $\Gamma\subseteq \Phi$ and $w \in W$, denote $w\Gamma:=\{w(\alpha) \mid \alpha \in \Gamma\}\subseteq \Phi$. 
A \emph{parabolic subsystem} is any subsystem of the form $w\Phi(J)$, likewise, a \emph{parabolic subgroup} of $W$ is any subgroup of the form $wW(J)w^{-1}$,
where $J \subseteq  \Delta$ and $w \in W$. 
 \begin{theorem}\label{thm:parabolic}
For $\Gamma \subseteq \Phi$, 
$\Phi \cap \mbox{span}(\Gamma)$ is a parabolic subsystem of $\Phi$ and its Weyl group is a parabolic subgroup of $W$.
\end{theorem}
\begin{proof} 
See, e.g., \cite[Lemma 3.2.3]{Kr94}, \cite[Proposition 2.6]{HRT97}.
\end{proof}

Recall the notation of $\Phi_X$, $\Phi^+_X$, $\Delta_X$ for $X \in L(\A)$ from Notation \ref{not:associated}. 
Note that if $X \in L_p(\A)$, then $\Phi_X$ is a parabolic subsystem of rank $p$ (Theorem \ref{thm:parabolic}).
For each $X \in L(\A)$, define the \textit{fixer} of $X$ by 
$$W_X:=\{w \in W \mid w(x)=x \mbox{ for all } x\in X\}.$$

\begin{proposition}\label{prop:Weyl arr}
If $X \in L(\A)$, then
\begin{enumerate}[(i)]
\item $W_X$ is the Weyl group of $\Phi_X$. Consequently, $W_X$ is a parabolic subgroup of $W$.
\item ${\A}_X$ is the Weyl arrangement of $\Phi_X^+$.
\end{enumerate}
\end{proposition} 
\begin{proof}
(i) The first statement follows from \cite[Proposition 2.5.5]{Ca72}. The second statement follows from Theorem \ref{thm:parabolic}. (ii) follows from (i) and \cite[Chapter V, $\S$3.3, Proposition 2]{B68}.
\end{proof}

\begin{definition} 
Two subsets $\Phi_1, \Phi_2\subseteq \Phi$ (resp., two subspaces $X_1, X_2\in L(\A)$) \emph{lie in the same $W$-orbit} if there exists $w\in W$ such that $\Phi_1 = w\Phi_2$ (resp., $X_1=w X_2$). 
Two subgroups $W_1, W_2$ of $W$ are \emph{$W$-conjugate} if there exists $w\in W$ such that $W_1 = w^{-1}W_2w$.
\end{definition}

\begin{lemma}\label{lem:compatible} 
Let $X_1,  X_2$ be subspaces in $L(\A)$. The following statements are equivalent:
\begin{enumerate}[(i)] 
\item $\Phi_{X_1}$ and $\Phi_{X_2}$ lie in the same $W$-orbit.
\item $\Delta_{X_1}$ and $\Delta_{X_2}$ lie in the same $W$-orbit.
\item $X_1$ and $X_2$ lie in the same $W$-orbit. 
\item $W_{X_1}$ and $W_{X_2}$ are $W$-conjugate.
\end{enumerate}
Consequently, if any one of the statements above holds, then $|\A^{X_1} | =|\A^{X_2}|$.
\end{lemma}
\begin{proof} 
The equivalence of the statements follows from \cite[Lemmas (3.4), (3.5)]{OS83} (see also \cite[Chapter VIII, 27-3, Proposition B]{Ka01}). 
The consequence is straightforward. 
\end{proof}


\section{Enumerate the cardinalities of $A_1^2$ restrictions}\label{sec:cardinality}

In this section, we present the first step towards proving conceptually the exponent formula in Theorem \ref{thm:combine}, the most important result in our paper.
When $X$ is of type $A_1^2$, we express the cardinality $|\A^X|$  in terms of the Coxeter number $\h$ and a certain sum of inner products of positive roots (Proposition \ref{prop:X=A_1^2}).

\begin{definition}\label{def:A_1^2}
A set $\{\beta_1,\beta_2\}\subseteq \Phi$ with $\beta_1\ne \pm\beta_2$ is called an \textit{$A_1^2$ set} if it spans a subsystem of type $A_1^2$, i.e.,
$\Phi\cap \mathrm{span}\{\beta_1,\beta_2\}=\{\pm\beta_1,\pm\beta_2\}$. 
\end{definition}

Thus, $X \in L(\A)$ is of type $A_1^2$ if and only if $\Delta_X$ is an $A_1^2$ set (Notation \ref{not:associated}). 

\begin{lemma}\label{lem:eg} 
 \quad
\begin{enumerate}[(i)]
\item For any $\beta=\sum_{\alpha \in \Delta}c_\alpha\alpha \in \Phi$, the set of $\alpha \in \Delta$ such that $c_\alpha \ne 0$  forms a non-empty connected induced subgraph of the Dynkin graph of $\Phi$. 
\item If $G$ is a non-empty connected subgraph of the Dynkin graph, then $\sum_{\alpha \in G}\alpha \in \Phi$.
\item If $\{\beta_1,\beta_2\} \subseteq \Delta$ and $(\beta_1,\beta_2)=0$, then $\{\beta_1,\beta_2\}$  is an $A_1^2$ set. 
\end{enumerate}
\end{lemma}
\begin{proof}
Proofs of (i) and (ii) can be found in \cite[Chapter VI, \S1.6, Corollary 3 of Proposition 19]{B68}. 
(iii) is an easy consequence of (i).
\end{proof}

Let $\T(A_1^2)$ (resp., $\T(RO)$) be the set consisting of $A_1^2$ (resp., RO) sets. 
\begin{proposition}\label{prop:characterize} \quad
\begin{enumerate}[(i)] 
\item
If $\{ \beta_1 ,\beta_2\}$ is $A_1^2$ (resp., RO), then $w\{ \beta_1 ,\beta_2\}$ is $A_1^2$ (resp., RO) for all $w \in W$. 
\item 
$\T(A_1^2)= \T(\Delta)$, where 
$\T(\Delta):=\{w\{\alpha_{i},\alpha_j\} \mid \{\alpha_{i},\alpha_j\} \subseteq \Delta,  (\alpha_i, \alpha_j)=0, w \in W\}$.
\end{enumerate}
 \end{proposition}
 \begin{proof}
(i) is straightforward. 
(ii) follows from (i), Theorem \ref{thm:parabolic} and Lemma \ref{lem:eg}(iii).
\end{proof}

\begin{remark}\label{rem:3}
By Proposition \ref{prop:characterize} and Definition \ref{def:RO}, $\T(A_1^2) \ne \emptyset$ (resp., $\T(RO)\ne \emptyset$) only when $\dim(V)\ge3$. 
\end{remark}

\begin{remark}\label{rem:numerical}
Only for giving additional information, we collect some numerical facts about $\T(A_1^2)$ and $\T(RO)$. 
These facts shall not be used in any of upcoming arguments that support the proof of Theorem \ref{thm:combine} or Theorem \ref{thm:basis-main}(i). 
Some of these facts will be rementioned in the proof of Theorem \ref{thm:basis-main}(ii) (Example \ref{ex:basis-RO}), which is the part we are unable to avoid the classification of all irreducible root systems.
By a direct check, $\T(RO)\cap\T(A_1^2)\ne\emptyset$ if and only if $\Phi=D_\ell$ with $\ell \ge 3$ ($D_3=A_3$). 
In general, $\T(RO) \setminus \T(A_1^2)\ne \emptyset$, for example when $\Phi=B_\ell$ ($\ell \ge 3$), $\{\epsilon_1-\epsilon_2,\epsilon_1+\epsilon_2\}$ is RO but spans a subsystem of type $B_2$ (notation in \cite{B68}).
There is only one orbit of $A_1^2$ sets, with the following exceptions:
\begin{enumerate}[(i)] 
\item when $\Phi=D_4$, $\T(RO)=\T(A_1^2)$, and there are three different orbits, 
\item when $\Phi=D_\ell$ $(\ell \ge 5)$, $\T(RO)  \subsetneq \T(A_1^2)$, and there are two different orbits: $\T(RO)=\{ w\{\epsilon_1-\epsilon_2,\epsilon_1+\epsilon_2\} \mid w \in W\}$ and $\T(A_1^2) \setminus \T(RO)$.
\item when $\Phi\in\{B_\ell, C_\ell\}$ $(\ell \ge 4)$, $\T(RO)\cap\T(A_1^2)=\emptyset$, and there are two different orbits: $\T(\Delta^{=}) :=\{ \{\alpha,\beta\} \in\T(A_1^2)  \mid \|\alpha\|= \|\beta\|\}$, and $\T(A_1^2) \setminus\T(\Delta^{=})$.
\end{enumerate}
\end{remark}

Recall the notation $\beta^\perp=\{\alpha \in \Phi \mid (\alpha,\beta)=0\}$ for $\beta \in \Phi$. 
 
\begin{definition}\label{def:N}  
For an $A_1^2$ set $\{ \beta_1, \beta_2\}\subseteq \Phi$, define
$$
\N_0=\N_0(\{ \beta_1, \beta_2\}):=\{\Psi \subseteq \Phi \mid \mbox{$\Psi$ is an irreducible subsystem of rank $3$, $\{ \beta_1, \beta_2\}\subseteq \Psi$}\},
$$
and for each $i \in \{1,2\}$, 
$$\M_{\beta_{3-i}}(\beta_i) :=
 \left\{ 
\Lambda \subseteq \beta_{3-i}^\perp\middle|
\begin{array}{c}
       \mbox{$\Lambda$ is an irreducible subsystem of rank $2$, $\beta_i \in  \Lambda$},  \\
\mbox{$\Phi\cap \mbox{span}(\{\beta_{3-i}\} \cup \Lambda)$ is a reducible subsystem of rank $3$}.
    \end{array}
\right\},$$
and 
$$
\N_{\beta_{3-i}}(\beta_i):=\{\Psi \subseteq \Phi \mid \Psi=\{\pm \beta_{3-i}\} \times \Lambda, \Lambda\in\M_{\beta_{3-i}}(\beta_i) \}.
$$

\end{definition} 
\begin{proposition}\label{prop:not empty}    
$\N_0$ is not empty.
\end{proposition} 
\begin{proof}
For $\beta_1, \beta_2 \in \Phi$, there exists $\delta \in \Phi$ such that $(\delta,\beta_1)\ne0$ and $(\delta,\beta_2)\ne0$, e.g., see \cite[Lemma 2.10]{HRT97}.
Thus $\Phi\cap \mbox{span}\{\beta_1, \beta_2, \delta\} \in \N_0$.
\end{proof}

In the remainder of this section, we assume that $X=H_{\beta_1} \cap H_{\beta_2}$ is an $A_1^2$ subspace  with $\beta_1, \beta_2 \in \Phi^+$.
If $Y \in \A^X$, then $\Phi_Y$ is a subsystem of rank $3$ and contains $\Delta_X=\{\beta_1, \beta_2\}$. 
\begin{proposition}\label{prop:=}
If $X \in L_{2}(\A)$ is an $A_1^2$ subspace, then 
\begin{align*}
\N_0 & =\{ \Phi_Y \mid \mbox{$Y \in \A^X$, $\Phi_Y$ is irreducible} \}, \mbox{and for each $i \in \{1,2\}$,} \\
\N_{\beta_{3-i}}(\beta_i) & = \{ \Phi_Y \mid \mbox{$Y \in \A^X$, $\Phi_Y=\{\pm \beta_{3-i}\} \times \Xi_Y$ and $\Xi_Y$ is irreducible of rank $2$}\}.
\end{align*}
\end{proposition} 
\begin{proof}
We only give a proof for the first equality. The others follow by a similar method.
Let $\Psi \in \N_0$. 
There exists $\delta\in \Phi^+$ such that $\Psi=\Phi\cap \mbox{span}\{\beta_1, \beta_2, \delta\}$.
Thus  $\Psi=\Phi \cap Y^{\perp}$ where $Y:=X\cap H_{\delta}$. 
Therefore $\Psi = \Phi_{Y}$ with $Y \in \A^X$.
\end{proof}

\begin{lemma}\label{lem:partition} 
Let $X \in L(\A)$. 
\begin{enumerate}[(i)]
\item $\A = \bigcup_{Y \in \A^X}\A_{Y}$.
\item If $Y, Y' \in \A^X$ and $Y \ne Y'$, then $\A_{Y} \cap \A_{Y'}=\A_{X}$.
\item $\A \setminus \A_X= \bigsqcup_{Y \in \A^X}(\A_{Y}\setminus \A_X)$ (disjoint union).
\end{enumerate}
 \end{lemma}

 \begin{proof} 
(i) is straightforward. For (ii), $\A_{X} \subseteq \A_{Y} \cap \A_{Y'}$  since $Y, Y' \subseteq X$. Arguing on the dimensions, and using the fact that $\dim(X)-\dim(Y)=1$ yield $X=Y+Y'$. Thus if $H \in \A_{Y} \cap \A_{Y'}$, then $X \subseteq H$, i.e., $H \in \A_X$. (iii) follows automatically from (i) and (ii).
\end{proof}

 \begin{corollary}\label{cor:passing}\quad
 Set $\N:=\N_0\bigsqcup \N_{\beta_2}(\beta_1)\bigsqcup \N_{\beta_1}(\beta_2)\bigsqcup \N_3$, where
 $$
\N_3= \N_3(\{ \beta_1, \beta_2\}):=\{\Psi \subseteq \Phi \mid \Psi=\{\pm \beta_1\}\times\{\pm \beta_2\} \times \{\pm \gamma\}, \gamma \in\Phi^+\}.
$$
For $\Psi \in \N$, set  $\Psi^+:=\Phi^+ \cap \Psi$. Then
$$\Phi^+ = \{\beta_1, \beta_2\} \sqcup \bigsqcup_{\Psi \in \N}(\Psi^+\setminus  \{\beta_1, \beta_2\}).$$
 \end{corollary}
 \begin{proof} 
It follows from Proposition \ref{prop:Weyl arr}(ii), Proposition \ref{prop:=} and Lemma \ref{lem:partition}.
\end{proof}

\begin{proposition}\label{prop:3sums}
For each $i \in \{1,2\}$, set
\begin{align*}
\K_0 & := \sum_{\Psi\in\N_0}\sum_{ \delta \in \Psi^+\setminus\{\beta_i\}}\left( \widehat{\beta_i}, \widehat\delta\right)^2,\\
\K_{\beta_{3-i}}(\beta_i) & := \sum_{\Lambda\in\M_{\beta_{3-i}}(\beta_i)}\sum_{ \delta \in \Lambda^+\setminus\{\beta_i\}}\left( \widehat{\beta_i}, \widehat\delta\right)^2.
\end{align*}
Then $2(\K_0+\K_{\beta_{3-i}}(\beta_i)+1)={\rm h}$. In particular, $\K_{\beta_2}(\beta_1)=\K_{\beta_1}(\beta_2)$.
\end{proposition}
 \begin{proof} 
By Corollary \ref{cor:passing},
$$2(\K_0+\K_{\beta_{3-i}}(\beta_i)+1)=2 \sum_{\delta \in \Phi^+}\left( \widehat{\beta_i}, \widehat\delta\right)^2,$$
which equals ${\rm h}$ by Theorem \ref{exponents}(vi).
\end{proof}

\begin{proposition}\label{prop:2 cases}   
$$|\Psi^+|=
\begin{cases}
3  \mbox{ if $\Psi \in \N_3$},\\
1+2 \sum_{\delta \in\Psi^+}\left( \widehat{\beta_i}, \widehat\delta\right)^2 \mbox{ if $\Psi \in \N_{\beta_{3-i}}(\beta_i) $}, i \in \{1,2\},\\
3\sum_{\delta \in \Psi^+}\left(\widehat{\beta_i}, \widehat\delta\right)^2 \mbox{ if $\Psi \in \N_{0}$}.
\end{cases}
$$
\end{proposition} 
\begin{proof} 
It follows from items (iv) and (vi) of Theorem \ref{exponents}.
\end{proof}

\begin{proposition}\label{prop:X=A_1^2} 
If $X=H_{\beta_1} \cap H_{\beta_2}$ is of type $A_1^2$, then for each $i \in \{1,2\}$
\begin{equation*}\label{eq:key}
|\A^{H_{\beta_i}}| - |\A^X|={\rm h}/2+\K_{\beta_{3-i}}(\beta_i).
\end{equation*}
\end{proposition}
\begin{proof}
The proof is similar in spirit to the proof of  \cite[Proposition 3.6]{OST86}.
By Corollary \ref{cor:passing},
\begin{equation*}\label{start}
\left|\A \right| -2=  \sum_{\Psi \in \N}\left( \left|\Psi^+ \right|-2\right).
\end{equation*}
It is not hard to see that $|\A^X|=|\N|$ (via the bijection $Y \mapsto \Phi_Y$).
By Proposition \ref{prop:2 cases},
\begin{align*}
\left|\A \right| -\left|\A^X \right| 
& =  \sum_{\Psi \in \N}\left( \left|\Psi^+ \right|-3\right)+2 \\
& = 3\K_0+2\K_{\beta_2}(\beta_1)+2\K_{\beta_1}(\beta_2)+2.
\end{align*}
By  Theorem \ref{thm:OST} and Proposition \ref{prop:3sums},
$$
|\A^{H_{\beta_i}}| - |\A^X| 
 = \left|\A \right| -\left|\A^X \right| - m_\ell  = {\rm h}/2 +\K_{\beta_{3-i}}(\beta_i). 
$$

\end{proof}
\begin{remark}\label{rem:K0}\quad
\begin{enumerate}[(i)] 
\item The conclusion of Proposition \ref{prop:X=A_1^2} can also be written as
\begin{equation*}\label{eq:also}
|\A^{H_{\beta_i}}| - |\A^X|=m_\ell-\K_0.
\end{equation*}
\item 
For each $\Psi\in\N_0$, denote by ${\rm h}(\Psi)$  the Coxeter number of  $\Psi$ and write $m_1(\Psi) \le m_2(\Psi)\le m_3(\Psi)$ for the exponents of $\Psi$. 
In fact, ${\rm{h}}(\Psi) =2m_2(\Psi)$ since $\mathrm{rank}(\Psi)=3$. Thus 
\begin{equation*}\label{eq:K0}
\K_0= \sum_{\Psi\in\N_0}\left( m_2(\Psi)-1 \right). 
\end{equation*}
In particular, if $\ell = 3$, then $\N_0=\{\Phi\}$ and $\K_0=m_2-1$. 
In this case,  $|\A^{H_{\beta_i}}| - |\A^X| =m_2=\h/2$.
\end{enumerate}
 \end{remark}
 
 \begin{remark}\label{rem:A1k} 
 If $\codim(X)>2$, for example, $X$ is of type $A_1^k$ with $k>2$, the calculation in Proposition \ref{prop:X=A_1^2} is expected to be more difficult (and harder to avoid classifications) as it involves the consideration on the rank $k+1$ subsystems of $\Phi$ containing $\Phi_X$.
  \end{remark}

\section{Proof of Theorem \ref{thm:combine}}\label{sec:main}

Theorem \ref{thm:combine} is a combination of two theorems below:
\begin{itemize}
\item Theorem \ref{thm:main}:  the freeness  part,
\item Theorem \ref{thm:card}:  the exponent part.
\end{itemize}
Since we prove the freeness part as a consequence of the exponent part and the proof is much more simple, we give the proof for the freeness part first (provided that  the exponent part is given).
Let us recall some general facts on hyperplane arrangements.
\begin{lemma}\label{lem:start} \quad
\begin{enumerate}[(i)]
\item  If $X, Y \in L(\A)$, then $(\A^X)^{X \cap Y}=\A^{X \cap Y}$. 
Similarly, $(\phi^X)^{X \cap Y}=\phi^{X \cap Y}$ for any $\phi \in D(\A)$.
\item If $H \in \A$ and $X \in L(\A^H)$, then $(\A^H)_X=(\A_X)^{H}$. 
We will use the notation $\A_X^H$ to denote these arrangements.
\end{enumerate}
 \end{lemma}
 \begin{proof} 
 Straightforward. 
\end{proof}
 
  \begin{lemma}\label{lem:decompose} 
Let $(\A_i, V_i)$ be irreducible arrangements ($1 \le i \le n$). 
Let $\A = \A_1 \times \cdots \times\A_n$ and $V= \oplus_{i=1}^n V_i$. 
Then
\begin{enumerate}[(i)]
\item  $\chi(\A, t)=\prod_{i=1}^n \chi(\A_i, t)$, 
\item for $H=V_1 \oplus  \cdots \oplus V_{k-1} \oplus H_k \oplus V_{k+1} \oplus  \cdots \oplus V_n \in \A$ ($H_k \in \A_k$), 
$$\A^H=\A_1 \times \ldots  \times \A_{k-1}  \times \A_k^{H_k}\times \A_{k+1}  \times \cdots \times \A_n.$$
\end{enumerate}
 \end{lemma}
\begin{proof}
(i) is a well-known fact, e.g.,  \cite[Lemma 2.50]{OT92}. (ii) follows by the definition of restriction. 
\end{proof}

The freeness of every restriction is settled by a case-by-case study in \cite{OT93}, and by a uniform method in \cite{D99}. 
In Theorem \ref{thm:main} below, we give a different and more direct proof for the freeness of $A_1^2$ restrictions by using the Addition-Deletion Theorem (Theorem \ref{thm:AD}) and the Combinatorial Deletion Theorem (Theorem \ref{thm:Abe}). 
We also need the following immediate consequence of Theorem \ref{thm:OST} and Theorem \ref{thm:card}.

\begin{proposition}\label{prop:weaker} 
If $\{\alpha,\beta\}$ is an $A_1^2$ set, then $|\A^{H_{\alpha}}| - |\A^{H_{\alpha} \cap H_{\beta}}|$ is a root of $\chi(\A^{H_{\alpha}}, t)= \prod_{i=1}^{\ell-1} (t-m_i)$.
 \end{proposition}

 \begin{theorem}\label{thm:main}
If $\{\alpha,\beta\}$ is an $A_1^2$ set, then the following statements are equivalent:
\begin{enumerate}[(1)]
\item $\A^{H_{\alpha} \cap H_{\beta}}$ is free and $\exp(\A^{H_{\alpha} \cap H_{\beta}})=\exp(\A)\setminus \{ m_i,m_\ell\} $ for some $i$ with $1 \le i \le \ell-1$. 
\item $|\A^{H_{\alpha}}| - |\A^{H_{\alpha} \cap H_{\beta}}|=m_i$.
\end{enumerate}
\end{theorem}
\begin{proof}
 It is sufficient to prove $(2) \Rightarrow (1)$. 
 By Theorem \ref{thm:OST}, $\A^{H_{\alpha}}$ is free with $\exp(\A^{H_{\alpha}})=\{ m_1,  \ldots, m_{\ell-1}\}$. 
By Lemma \ref{lem:start}, $\A^{H_{\alpha} \cap H_{\beta}}=(\A^{H_{\alpha}})^{H_{\alpha}\cap H_{\beta}}$.
 If we can prove that $\A^{H_{\alpha}} \setminus \{H_{\alpha}\cap H_{\beta}\}$ is free, then Theorem \ref{thm:AD} and Condition (2) guarantee that $\A^{H_{\alpha} \cap H_{\beta}}$ is free with $\exp(\A^{H_{\alpha} \cap H_{\beta}})=\exp(\A)\setminus \{ m_i,m_\ell\} $ for some $1 \le i \le \ell-1$. 

To show the freeness of $\A^{H_{\alpha}} \setminus \{H_{\alpha}\cap H_{\beta}\}$, we use Theorem \ref{thm:Abe}. We need to show that $|\A^{H_{\alpha}}_X| - |\A^{H_{\alpha} \cap H_{\beta}}_X|$ is a root of $\chi(\A^{H_{\alpha}}_X, t)$ for all $X \in L(\A^{H_{\alpha} \cap H_{\beta}})$. 
It is clearly true by Proposition \ref{prop:weaker} provided that $\A_X$ is irreducible. 
If $\A_X$ is reducible, write $\A_X =  \A_1 \times \cdots \times\A_n$ where each $\A_i$ is irreducible. 
By Lemma \ref{lem:decompose},  $|\A^{H_{\alpha}}_X| - |\A^{H_{\alpha} \cap H_{\beta}}_X|$ either equals $|\A^{H_k}_k| - |\A^{H_k \cap H'_k}_k|$ where $1 \le k \le n$, $H_k, H'_k \in \A_k$, $H_k \cap H'_k$ is $A_1^2$ with respect to $\A_k$, or equals $|\A_j| - |\A_j^{H_j}|$ for some $j \ne k$. 
In either case, $|\A^{H_{\alpha}}_X| - |\A^{H_{\alpha} \cap H_{\beta}}_X|$ is a root of $\chi(\A^{H_{\alpha}}_X, t)=\chi(\A_1, t)\ldots\chi(\A^{H_k}_k, t)\ldots\chi(\A_j, t)\ldots\chi(\A_n, t)$.
\end{proof}

\begin{remark}\label{rem:A2,B2}
If the subsystem $\Phi_{\{\alpha,\beta\}}\subseteq \Phi$ spanned by $\{\alpha,\beta\}$ is not of type $A_1^2$, then Condition (2) in Theorem  \ref{thm:main} may not occur, i.e., the number $|\A^{H_{\alpha}}| - |\A^{H_{\alpha} \cap H_{\beta}}|$ may not be an exponent of $W$. 
In this case, we are unable to use the Addition-Deletion Theorem (Theorem \ref{thm:AD}) to derive $\exp(\A^{H_{\alpha} \cap H_{\beta}})$.
To see an example, let $\Phi=E_6$ and $\Phi_{\{\alpha,\beta\}}$ be of type $A_2$, or let $\Phi=F_4$ and $\Phi_{\{\alpha,\beta\}}$ be of type $B_2$ (see, e.g., \cite[Tables V and VI]{OS83}). 
\end{remark}

 \begin{theorem}\label{thm:card} 
If $\{\beta_1, \beta_2\} \subseteq \Phi^+$ is $A_1^2$, then for each $i \in \{1,2\}$
$$|\A^{H_{\beta_i}}| - |\A^{H_{\beta_1} \cap H_{\beta_2}}|=
\begin{cases}
\h/2 \mbox{ if $\{\beta_1, \beta_2\}$ is RO}, \\
m_{\ell-1} \mbox{ if $\{\beta_1, \beta_2\}$ is not RO}.
\end{cases}
$$
 \end{theorem}
 
Proof of Theorem   \ref{thm:card} will be divided into two halves: Theorems \ref{thm:half1} and \ref{thm:half2}. 


\subsection{First half of Proof of Theorem \ref{thm:card}}\label{sec:halfway1}

\begin{lemma}\label{lem:Cartan} 
 Let $C=(c_{ij})$ with $c_{ij} = 2(\alpha_i, \alpha_j)/(\alpha_j, \alpha_j)$ be the Cartan matrix of $\Phi$. The roots of the characteristic polynomial of $C$ are $2+2\cos(m_i\pi/{\rm h})$ ($1 \le i \le \ell$).
 \end{lemma}
 \begin{proof} 
See, e.g.,  \cite[Theorem 2]{BLM89}.
\end{proof}

We denote by ${\D}(\Phi)$ the Dynkin graph and by $\widetilde{\D}(\Phi)$ the extended Dynkin graph of $\Phi$.
 
\begin{theorem}\label{thm:half1} 
Assume that there exists a set $\{\beta_1, \beta_2\} \subseteq \Phi^+$ such that $\{\beta_1, \beta_2\}$ is both $A_1^2$ and RO. 
Then $\h/2$ is an exponent of $W$. 
Moreover, for each $i \in \{1,2\}$, 
$$|\A^{H_{\beta_i}}| - |\A^{H_{\beta_1} \cap H_{\beta_2}}|= \h/2.$$
 \end{theorem} 
\begin{proof} 
By Lemma \ref{lem:Cartan}, it suffices to prove that the characteristic polynomial of the Cartan matrix $C$ admits $2$ as a root. 
By Proposition \ref{prop:characterize}, we may assume that $\{\beta_1, \beta_2\} \subseteq \Delta$.
Since the Dynkin graph $\D(\Phi)$ of $\Phi$ is a tree, there is a unique path in $\D(\Phi)$ which admits  $\beta_1$ and $\beta_2$ as endpoints. 
By the definition of RO sets, this path contains exactly one more vertex of $\D(\Phi)$, say $\beta_3$ with $(\beta_1, \beta_3)\ne0$ and  $(\beta_2, \beta_3)\ne0$.
Also, the other vertices of $\D(\Phi)$, if any, connect to the path only at $\beta_3$. 
The Cartan matrix has the following form:
$$
C =\begin{bmatrix}
2 & 0  & c_{13}  & 0 &  \dotsb &  & 0\\
0 & 2 &  c_{23}  & 0 &  \dotsb &  & 0\\
c_{13} & c_{23}  & 2  &  &  &  & \\
0  & 0 &   &  & \\
\vdots &\vdots &    &   & \ddots  &  & \\
0 & 0 &    &   &  & 2 & \\
0  & 0 &        &   &  &  & 2
\end{bmatrix}.
$$
By applying the Laplace's formula, it is easily seen that $\det(xI_\ell- C)$ is divisible by $x-2$.

Assume that $\M_{\beta_2}(\beta_1) \ne \emptyset$ and let $\Lambda\in\M_{\beta_2}(\beta_1)$ (notation in Definition \ref{def:N}). 
Since $\{\beta_1, \beta_2\}$ is RO, the fact that $\Lambda \subseteq \beta_2^\perp$ implies that $(\beta_1, \alpha)=0$ for all $ \alpha \in\Lambda \setminus \{\pm\beta_1\}$. 
This contradicts the irreducibility of $\Lambda$. 
Thus $\K_{\beta_2}(\beta_1)=0$. 
Proposition \ref{prop:X=A_1^2}  completes the proof.
 \end{proof}
\subsection{Root poset, Dynkin diagram and exponents}\label{subsec:elemetary}
The second half of Proof of Theorem \ref{thm:card} is complicated but uses nothing rather than combinatorial properties of root systems. 
We need to recall and prove several (technical) properties of the root poset, Dynkin diagram and exponents. 
This section is devoted to doing so, and every statement will be provided in great detail.

First, we study the positive roots of height $\ge m_{\ell-1}$ in the root poset and a certain set of vertices  in the extended Dynkin diagram.
For $\alpha \in V$, $\beta \in V\setminus\{0\}$,  denote $\langle \alpha,\beta \rangle := \frac{2( \alpha,\beta ) }{( \beta , \beta)}$. 
Let $\theta:= \sum_{i=1}^\ell  c_{\alpha_i}\alpha_i$ be the highest root of $\Phi$, and we call $c_{\alpha_i} \in \Z_{>0}$ the coefficient of $\theta$ at the simple root $\alpha_i$. 
Denote by $c_{\max}:=\max \{c_{\alpha_i}\mid1 \le i\le \ell\}$ the largest coefficient. 

\begin{proposition}\label{lem:coes} 
Let $\Phi$ be an irreducible root system in $\R^\ell$. 
Let  $\theta$ be the highest root of $\Phi$, and denote $\lambda_0 := -\theta$, $c_{\lambda_0}:=1$.
Suppose that the elements of a fixed base $\Delta=\{\lambda_1, \ldots, \lambda_\ell\}$ are labeled so that $\Lambda:=\{\lambda_0, \lambda_1, \ldots, \lambda_q\}$ is a set of minimal cardinality such that $c_{\max}=c_{\lambda_q}$ and $(\lambda_s,\lambda_{s+1}) < 0$ for $0 \le s \le q -1$. 
\begin{enumerate}[(i)] 
\item Then $c_{\lambda_s}=s + 1$ for $0\le s \le q$ and $|\Lambda|=c_{\max}$.
\item Assume that  $c_{\max}\ge 2$. Then $(\lambda_0,\lambda_1, \ldots, \lambda_{q-1})$ is a simple chain of $\widetilde{\D}(\Phi)$ connected to the other vertices only at $\lambda_{q-1}$. 
\end{enumerate}
\end{proposition}

 \begin{proof} 
See, e.g., \cite[Lemma B.27, Appendix B]{MT11} and \cite[Proposition 3.1]{Tan17}.
\end{proof}

\begin{remark}\label{rem:resolve}
Proposition \ref{lem:coes} was first formulated and proved in terms of coroots in \cite[Lemma 1.5]{R75} under the name \emph{lemma of the string}. 
The proof of \cite[Lemma B.27, Appendix B]{MT11} contains a small error, which was resolved in \cite[Proposition 3.1]{Tan17}.
 \end{remark}

 \begin{corollary}\label{cor:c-max}\quad
 \begin{enumerate}[(i)]
\item If $c_{\max}=1$, then all roots of $\Phi$ have the same length. 
In addition, if $\ell \ge 2$, then ${\D}(\Phi)$ is a simple chain and $-\theta$ is connected only to two terminal vertices of ${\D}(\Phi)$.
\item  If $c_{\max}\ge 2$, then $-\theta$ is connected only to one vertex $\lambda$ of ${\D}(\Phi)$ with $\lan\theta,\lambda \ran\in \{1,2\}$ and $c_{\lambda}=2$. 
In particular, if $\lambda_i \in \Delta$ with $c_{\lambda_i}=1$ is connected to $\lambda_j \in \Delta$ with $c_{\lambda_j}\ge 2$, then $\lambda_i$ must be a terminal vertex of ${\D}(\Phi)$ and $\lan\lambda_j,\lambda_i\ran=-1$.
\end{enumerate}
\end{corollary}
 \begin{proof} 
The second statement of (ii) follows from the equation $\lan\theta,\lambda_i \ran=0$.
The other statements can be found in \cite[Proposition 3.1 and Remark 3.2]{Tan17}.
\end{proof}

\begin{corollary}\label{cor:not-important}
Assume that $c_{\max}\ge 2$. Either $\lan\lambda_{q-1},\lambda_{q}\ran\in\{-2,-3\}$ or $\lambda_{q}$ is a ramification point of $\widetilde{\D}(\Phi)$.
\end{corollary}
\begin{proof}
See  \cite[Corollary 3.3]{Tan17}. 
\end{proof}

Denote $\U:=\{\theta_i\in \Phi^+ \mid {\rm ht}(\theta_i) >m_{\ell-1}\}$, and set $m:=|\U|$. 
By Theorems \ref{thm:dual} and \ref{exponents}(i), (iii), we have $m=m_{\ell}-m_{\ell-1}=m_{2}-1>0$.
Suppose that the elements of $\U$ are labeled so that $\theta_1$ denotes the highest root, and $\xi_i=\theta_i-\theta_{i+1} \in\Delta$ for $1 \le i \le m-1$. 
We also adopt a convention $\xi_0:=-\theta_1$. 
Set $\Xi:=\{\xi_i \mid 0 \le i \le m-1\}$.
Note that $\Xi$ is a multiset, not necessarily a set. 
For a finite multiset $S=\{(a_1)^{b_1},\ldots, (a_n)^{b_n}\}$, write $\overline{S}$ for the base set of $S$, i.e., $\overline{S}=\{a_1,\ldots, a_n\}$.
Let us call 
\begin{enumerate}
\item[Case 1:]  ``there is an integer $t$ such that $1 \le t \le m-1$ and $\lan\theta_t,\xi_t \ran= 3$".
\item[Case 2:]  Negation of Case $1$.
\end{enumerate}
\begin{proposition}\label{prop:b-a}
\quad
\begin{enumerate}
\item[(i)]  If Case $1$ occurs, then $t=m-2$ and $\Xi=\{\xi_0, \xi_1, \ldots, (\xi_{m-2})^2\}$ with $\xi_i \ne \xi_j$ for $0 \le i < j \le m-2$. 
As a result, $|\overline{\Xi}|=m_2-2$. 
\item[(ii)]   If Case $2$ occurs,  then  $\Xi=\{\xi_0, \xi_1, \ldots, \xi_{m-1}\}$ with $\xi_i \ne \xi_j$ for $0 \le i < j \le m-1$. 
As a result, $|\Xi|=m_2-1$. 
\end{enumerate}
\end{proposition}

 \begin{proof} 
 See \cite[Propositions 3.9 and 3.10]{Tan17}.
 \end{proof}

 \begin{theorem}\label{thm:iso}
With the notations we have seen from Proposition \ref{lem:coes} to Proposition \ref{prop:b-a}, we have that $q=m-1$ and $\lambda_i = \xi_i$ for all $1 \le i \le q$. 
In particular, $\overline{\Xi}=\Lambda$ (as sets). 
Moreover, if Case $1$ occurs, then $m_2 = c_{\max}+2$; and if Case $2$ occurs, then $m_2 = c_{\max}+1$.
\end{theorem}
 \begin{proof} 
 See  \cite[Theorem 4.1]{Tan17}.
 \end{proof}
 
 \begin{corollary}\label{cor:differences}
 \quad
\begin{enumerate}
\item[(i)] If Case 1 occurs, then $\theta_i-\theta_j \in \Phi^+$ for $1 \le i<j \le m$, $\{i,j\} \ne \{m-2,m\}$, and $\theta_{m-2}-\theta_m \in 2\Delta$.
\item[(ii)] If Case 2 occurs, then $\theta_i-\theta_j \in \Phi^+$ for $1 \le i<j \le m$.
\end{enumerate}
\end{corollary}
\begin{proof}
 See  \cite[Corollary 3.11]{Tan17}.
\end{proof}

 \begin{corollary}\label{cor:criterion}
 The following statements are equivalent: (i) Case $1$ occurs, (ii) $\rank(\Phi)=2$, (iii) $\Phi=G_2$, (iv) $c_{\max}=m_2-2$.
\end{corollary}
 \begin{proof} 
 See  \cite[Theorem 4.2]{Tan17}.
 \end{proof}
  Recall the notation $\Theta^{(r)}=\{\alpha \in \Phi^+ \mid {\rm ht}(\alpha)=r\}$.

\begin{proposition}\label{prop:long}
There is always a long root in $\Theta^{(m_{\ell-1})}$.
\end{proposition}
 \begin{proof} 
We may assume that  $c_{\max}\ge 2$.
If not,  by Corollary \ref{cor:c-max}, all roots of $\Phi$ have the same length. The assertion is trivial. 
 If Case $1$ occurs,  then by Corollary \ref{cor:criterion}, $\Phi=G_2$. 
 The assertion is also trivial. 

Now we can assume that Case $2$ occurs. 
By  Proposition \ref{prop:b-a}, $\xi_i \ne \xi_j$ for $0 \le i < j \le m-1$. 
If $\lan \xi_{m-2},\xi_{m-1}\ran=-2$, then $\lan \theta_{m-1},\xi_{m-1}\ran=2$.
Thus $\theta_{m-1}-2\xi_{m-1} \in \Theta^{(m_{\ell-1})}$, that is a long root and we are done. 
We are left with the case $\lan \xi_{m-2},\xi_{m-1}\ran=-1$. 
By Corollary \ref{cor:not-important}, $\xi_{m-1}$ is connected to at least two vertices of $\widetilde{\D}(\Phi)$ apart from $\xi_{m-2}$, say $\mu_1,\ldots, \mu_k$ ($k\ge2$). 

We claim that there exists $\mu_i$ such that $\lan \xi_{m-1},\mu_i\ran=-1$. 
Proof of the claim when $m=2$ (i.e., $c_{\max}=2$) and $m\ge3$ uses very similar technique (the case $m=2$ is actually Lemma \ref{lem:cases}(2)).
We only give a proof when $m\ge3$.
From $\lan\theta, \xi_{m-1}\ran =0$ (it equals $1$ if $m=2$), $c_{\xi_{m-1}}=c_{\max}$, $c_{\xi_{m-2}}=c_{\max}-1$, and $\xi_{m-1}$ is a long root, we have $c_{\max}+1-\sum_{i=1}^k c_{\mu_i}=0$. 
Suppose to the contrary that $\lan \xi_{m-1},\mu_i\ran\le-2$ for all $1 \le i \le k$. 
From $0=\lan\theta, \mu_i\ran \le 2c_{\mu_i}+c_{\max}\lan\xi_{m-1}, \mu_i\ran$, we obtain $c_{\mu_i}=c_{\max}$. 
Thus, $c_{\max}+1-kc_{\max}=0$, a contradiction. 
So we can choose $\mu_i$ so that $\lan \xi_{m-1},\mu_i\ran=-1$. 
Therefore, $\lan \theta_m, \mu_i\ran=-\lan \xi_{m-1},\mu_i\ran=1$ and $\theta_{m}-\mu_i \in \Theta^{(m_{\ell-1})}$, that is a long root. 
\end{proof}

\begin{lemma}\label{lem:3roots} 
Suppose $\beta_1,  \beta_2, \beta_3 \in \Phi$ with $\beta_1+  \beta_2+ \beta_3 \in \Phi$ and $\beta_i+  \beta_j \ne 0$ for $i \ne j$. Then at least two of the three partial sums $\beta_i+  \beta_j$ belong to $ \Phi$.
\end{lemma}
\begin{proof}
 See, e.g., \cite[\S11, Lemma 11.10]{LN04}.
\end{proof}

\begin{proposition}\label{prop:irr}
If $\gamma\in \Theta^{(m_{\ell-1})}$, then $\theta_i - \gamma \in k\Phi^+$ with $k\in\{1,2,3\}$ for every  $1 \le i \le m$. 
\end{proposition}
 \begin{proof} 
To avoid the triviality, we assume that Case $2$ occurs and $1 \le i \le m-1$. 
Denote $\mu:= \theta_m -\gamma \in \Delta$. 

Assume that $ \mu=\xi_{m-1}$. Then $\theta_{m-1}- \gamma=2\xi_{m-1}  \in 2\Delta$. 
We have $\langle\theta_{m-1},\xi_{m-1} \rangle =\langle\gamma,\xi_{m-1} \rangle  +4$. 
It follows that  $\langle\theta_{m-1},\xi_{m-1} \rangle =2$.
Fix $i$ with $1 \le i \le m-2$, and set $\alpha:= \theta_i -\theta_{m-1}\in \Phi^+$ (by Corollary \ref{cor:differences}). 
Since $ \theta_i = \theta_1 - (\xi_1+\cdots+\xi_{i-1})$, we have $\langle\theta_i,\xi_{m-1} \rangle =0$. Thus $\langle\alpha ,\xi_{m-1} \rangle=-\langle\theta_{m-1},\xi_{m-1} \rangle =-2$. 
Then $\theta_i -\gamma =\alpha +2\xi_{m-1}\in \Phi^+$.
 
Assume that $ \mu\ne\xi_{m-1}$. Then $\theta_{m-1}= \gamma+\mu+\xi_{m-1}$. 
By Lemma \ref{lem:3roots}, $\mu+\xi_{m-1}\in \Phi^+$ since $\gamma+\xi_{m-1}\notin \Phi^+$. 
Thus $\mu$ and $\xi_{m-1}$ are adjacent on ${\D}(\Phi)$. 
If $\mu \ne \xi_{m-2}$ (of course, $\mu \ne \xi_i$ for all $1 \le i \le m-3$ since ${\D}(\Phi)$ is a tree), 
then by Lemma \ref{lem:eg}(ii), $\theta_i -\gamma =\xi_i+\cdots+\xi_{m-1}+\mu\in \Phi^+$ for each $1 \le i \le m-1$. 
If $\mu = \xi_{m-2}$, then $\theta_{m-2}= \gamma+\xi_{m-1}+2\xi_{m-2}$. 
Thus $\langle\theta_{m-2},\xi_{m-2} \rangle =\langle\gamma,\xi_{m-2} \rangle  +\langle\xi_{m-1} ,\xi_{m-2} \rangle+4$. 
Using the fact that $\xi_{m-2}$ is a long root, we obtain a contradiction since the left-hand side is at most $1$ while the right-hand side is at least $2$.
 \end{proof}

 \begin{corollary}\label{cor:exactly2}
If $\ell \ge 5$, then $m_{\ell-2}<m_{\ell-1}$.
\end{corollary}
 
 \begin{proof} 
By Theorem \ref{thm:dual}, it suffices to prove that there are exactly two roots of height $m_{\ell-1}$, i.e., $|\Theta^{(m_{\ell-1})}|=2$. To avoid the triviality, we assume that Case $2$ occurs and $c_{\max}\ge 2$. 
We need to consider two cases: $\lan \xi_{m-2},\xi_{m-1}\ran=-2$ or $\lan \xi_{m-2},\xi_{m-1}\ran=-1$. Since the proofs are very similar, we only give a proof for the latter (slightly harder case). 

Suppose to the contrary that $\Theta^{(m_{\ell-1})}=\{\gamma_1,\ldots,\gamma_k\}$ with $k \ge 3$. 
By Proposition \ref{prop:irr}, $\theta_{m-1}-\gamma_i =\xi_{m-1}+\mu_i \in \Phi^+$, where $\mu_i \in \Delta$ ($1 \le i \le k$). 
Thus $\mu_i$ is adjacent to $\xi_{m-1}$ on $\widetilde{\D}(\Phi)$.
By the same argument as the one used in the end of Proof of Proposition \ref{prop:irr}, we obtain $\mu_i \ne \xi_{m-2}$ for all $1 \le i \le k$. 
If $m\ge3$, then the same argument as in Proof of Proposition \ref{prop:long} gives $c_{\max}+1=\sum_{i=1}^k c_{\mu_i}$.
From $0=\lan\theta, \mu_i\ran \le 2c_{\mu_i}-c_{\max}$, we obtain $c_{\mu_i}\ge c_{\max}/2$ for all $1 \le i \le k$. 
This forces $k=3$. 
But it implies that $c_{\max}\le2$, i.e., $m\le2$, a contradiction. 
Now consider $m=2$. 
A similar argument as above shows that $k=3$ and $c_{\mu_1}=c_{\mu_2}=c_{\mu_3}=1$. 
The second statement of Corollary \ref{cor:c-max}(ii) implies that $c_{\mu_i}$ must be all terminal. 
Thus $\ell=4$, a contradiction. 
\end{proof}

From now on, we require the classification of root systems of rank $\le 4$ to make some arguments work. The classification of rank $3$ or $4$ root systems will be announced before use, while that of rank $2$ root systems (has been and) will be used without announcing.

\begin{lemma}\label{lem:repeat} 
If $m_2={\rm h}/2$, then $\ell \le 4$. 
More specifically, when $\ell=4$, $m_2={\rm h}/2$ if and only if $\Phi=D_4$.
 \end{lemma}
 \begin{proof} 
If $m_2 = c_{\max}+2$, then by Corollary \ref{cor:criterion}, $\Phi=G_2$. 
However, $m_2>{\rm h}/2$ by a direct check.
Now assume that $m_2 = c_{\max}+1$. 
Recall from Proposition \ref{lem:coes} that the coefficients of $\theta$ at elements of $\Lambda=\{\lambda_0, \lambda_1, \ldots, \lambda_q\}$ form an arithmetic progression, starting with  $c_{\lambda_0}=1$ and ending with $c_{\lambda_q}= c_{\max}=q+1$.
From $\sum_{\lambda \in \Lambda}c_{\lambda}+\sum_{\lambda \in \Delta\setminus \Lambda}c_{\lambda} = \h$, we have
$$c_{\max}( c_{\max}+1)/2+ \ell-(c_{\max}-1) \le 2c_{\max}+2.$$
Thus $\ell \le (-c_{\max}^2+5c_{\max} +2)/2$. Therefore, $\ell \le 4$.
The second statement is clear from the classification of irreducible root systems of rank $4$.
\end{proof}

\begin{remark}\label{rem:well-known}
The first statement of Lemma \ref{lem:repeat} is an easy consequence of the well-known fact that every exponent of $\Phi$ appears at most twice. A uniform proof of this fact is probably well-known among experts.
 \end{remark}
 
 Now we investigate the ``local" picture of the extended Dynkin graph at the subgraph induced by the negative $-\theta$ of the highest root  and the simple root adjacent to it, and find a connection with RO properties. 

\begin{lemma}\label{lem:cases} 
Assume that $\ell \ge 2 $ and $c_{\max}\ge 2$. 
Let $\lambda$ be the unique simple root connected to $-\theta$ (Corollary \ref{cor:c-max}). 
Denote by $\gamma_1,\ldots, \gamma_k$ ($k\ge1$) the simple roots connected to $\lambda$. 
Then there are the following possibilities:
 \begin{enumerate}[(1)]
\item If  $\lan\theta,\lambda \ran=2$, then $k=1$, $c_{\gamma_1}\in \{1,2\}$.
\item If  $\lan\theta,\lambda \ran=1$ (in particular, $\lambda$ is long), then either $(2a)$ $k=3$, $c_{\gamma_1}=c_{\gamma_2}=c_{\gamma_3}=1$ (i.e., $\Phi=D_4$), or $(2b)$ $k=2$, $c_{\gamma_1}=2$, $c_{\gamma_2}=1$ ($\gamma_2$ is terminal and long), or $(2c)$ $k=1$, $c_{\gamma_1}=3$.
\end{enumerate}
 \end{lemma}

 \begin{proof} 
We have $\lan\theta,\lambda\ran =2c_{\lambda}+ \sum_{i=1}^k c_{\gamma_i}\lan\gamma_i,\lambda\ran \le 4- \sum_{i=1}^k c_{\gamma_i}$. 
Since $\lan\theta,\lambda \ran \ge 1$, we have $\sum_{i=1}^k c_{\gamma_i} \in \{1,2,3\}$.
Then we can list all possibilities and rule out impossibilities. 
For example, if $\lan\theta,\lambda \ran=1$ and $\sum_{i=1}^k c_{\gamma_i} =2$, then either (i) $k=1$, $c_{\gamma_1}=2$, or (ii) $k=2$, $c_{\gamma_1}=c_{\gamma_2}=1$. 
For (i), it follows that $1=\lan\theta,\lambda\ran =4+2\lan\gamma_i,\lambda\ran$, which is a contradiction since $\lambda$ is long. 
For (ii), the second statement of Corollary \ref{cor:c-max}(ii) implies that $\gamma_1,\gamma_2$ must be all terminal. 
Thus $\ell=3$. 
However, such root system does not exist by the classification of irreducible root systems of rank $3$.
Similarly, to conclude that $\Phi=D_4$ in (2a), we need the classification of irreducible root systems of rank $4$.
\end{proof}

It is known that $\theta^\perp$ is the standard parabolic subsystem of $\Phi$ generated by $\{\alpha \in \Delta \mid (\alpha,\theta)=0\}$. 
Also, $\theta^\perp$ may be reducible and decomposed into  irreducible, mutually orthogonal components. 

\begin{corollary}\label{cor:reducible}
If $\theta^\perp$ is reducible, then either Possibility (2a) or (2b) occurs. 
\end{corollary}
\begin{proof}
It follows immediately from Lemma \ref{lem:cases} (taking $k\ge 2$).
\end{proof}

\begin{corollary}\label{cor:existence-D4}
When $\ell=4$, there exists a set that is both $A_1^2$ and RO if and only if $\Phi=D_4$.
Moreover, if  $\Phi=D_4$, then every $A_1^2$ set is RO.
\end{corollary}
\begin{proof}
Use the classification of irreducible root systems of rank $4$.
\end{proof}

\begin{proposition}\label{prop:non-A1}
Assume that $\ell \ge 4$.
If $\theta^\perp$ is reducible and there exists an $A_1^2$ set that is not RO, then Possibility (2b) in Lemma \ref{lem:cases} occurs. 
In particular, $\theta^\perp = \{\pm\gamma_2\} \times\Omega$ for a long simple root $\gamma_2$ and $\Omega$ is irreducible with $\mathrm{rank}(\Omega)\ge 2$.
\end{proposition}
 \begin{proof} 
 If $c_{\max}=1$, then  Corollary \ref{cor:c-max}(i) implies that $\theta^\perp$ is irreducible and of rank at least $2$, a contradiction.
Now consider $c_{\max}\ge 2$. 
By Corollary \ref{cor:reducible}, either Possibility (2a) or (2b) occurs. 
However, Corollary \ref{cor:existence-D4} ensures that Possibility (2a) can not occur because a non-RO $A_1^2$ set exists. 
Thus Possibility (2b) must occur. 
Then $\theta^\perp = \{\pm\gamma_2\} \times\Omega$ where $\Omega$ is irreducible with $\mathrm{rank}(\Omega)\ge 2$.
\end{proof}
\begin{proposition}\label{prop:A1-component}
Assume that $\ell \ge 4$.
Suppose that $\{\theta, \alpha\}$ is an $A_1^2$ set with $\alpha \in \Delta$. 
If $\{\pm\alpha\}$ is a component of $\theta^\perp$, then $\{\theta, \alpha\}$ is RO.

\end{proposition}
 \begin{proof} 
We may assume that  $c_{\max}\ge 2$.
If not, Corollary \ref{cor:c-max}(i) implies that $\theta^\perp$ is irreducible and of rank at least $2$, a contradiction. 
Since $-\theta$ connects only to one vertex of ${\D}(\Phi)$, $\theta^\perp$ must be reducible; otherwise, $\theta^\perp=\{\pm\alpha\}$, a contradiction.
Thus, $\{\beta, \alpha\}$ is $A_1^2$ for every $\beta \in \theta^\perp\setminus \{\pm \alpha\}$ because $\{\beta, \alpha\}=w\{\alpha', \alpha\}$ for some $w \in W$ and $\alpha' \in \Delta$ such that $(\alpha', \alpha)=0$. 
With the notations in Definition \ref{def:N} and Proposition \ref{prop:3sums}, we have $\M_\theta(\alpha)=\emptyset$ and $\K_\theta(\alpha)=0$.

Note that $\theta^\perp\setminus \{\pm \alpha\} \subseteq  \alpha^\perp\setminus \{\pm\theta\}$ since $\{\pm\alpha\}$ is a component of $\theta^\perp$.
Suppose to the contrary that $\{\theta, \alpha\}$ is not RO. 
Then there exists $\beta \in \alpha^\perp\setminus \{\pm\theta\}$ such that $(\beta, \theta)\ne0$. 
Note that $\Gamma:=\Phi\cap \mathrm{span}\{\beta, \theta, \alpha\}$ is a subsystem of rank $3$ since it contains the $A_1^2$ set $\{\theta, \alpha\}$. 
$\Gamma$ must be irreducible, otherwise, $\Gamma\in \M_\alpha(\theta)$ and $\K_\alpha(\theta)\ne0$ which contradicts Proposition \ref{prop:3sums}. 
Relying on the classification of rank-$3$ irreducible root systems, and two facts: (i) $\{\theta, \alpha\}$ is $A_1^2$ in $\Gamma$, (ii) $\alpha^\perp \cap \Gamma$ is an irreducible subsystem of rank $2$ of $\Gamma$ (as it contains $\beta$ and $\theta$), we conclude that $\Gamma = B_3$ and $\alpha$ is the unique short simple root of $\Gamma$. 
Since $\theta^\perp$ is reducible, either Possibility (2a) or (2b) occurs by Corollary \ref{cor:reducible}.
Thus Possibility (2b) must occur since $\Phi$ contains the subsystem $\Gamma$ of type $B_3$.
Since $\{\pm\alpha\}$ is a component of $\theta^\perp$ and $\alpha$ is short, with the notation in Possibility  (2b), $\Delta=\{\gamma_2, \lambda, \alpha\}$. 
Thus $\ell=3$, a contradiction. 
\end{proof}

\subsection{Second half of Proof of Theorem \ref{thm:card}}\label{subsec:halfway2}
In this subsection, we complete the Proof of Theorem \ref{thm:card} by proving its second half, Theorem \ref{thm:half2} below.
\begin{theorem}\label{thm:half2} 
If $\{\beta_1, \beta_2\} \subseteq \Phi^+$ is $A_1^2$ but not RO, then for each $i \in \{1,2\}$
$$|\A^{H_{\beta_i}}| - |\A^{H_{\beta_1} \cap H_{\beta_2}}|= m_{\ell-1}.$$
 \end{theorem} 
 
The first key ingredient is Theorem \ref{thm:crucial}. 
It asserts that Theorem \ref{thm:half2} is always true for a special class of $A_1^2$ sets, and we can describe it by the height function without requiring the RO condition. 
We were recently informed that M\"{u}cksch and R\"{o}hrle \cite{MR20} study a similar property (to the non-RO case) of Weyl arrangement restrictions (which they called the accuracy) via MAT-free techniques of \cite{ABCHT16}.
Their main result together with Corollary \ref{cor:exactly2} indeed give a proof of Theorem \ref{thm:crucial}. 
We will use this proof here, however, we remark that our primary method gives a different and bijective proof, which we refer the interested reader to Appendix \S \ref{sec:app} for more details.

Unless otherwise stated, we assume that $\ell \ge 3$ in the remainder of the paper.

 \begin{lemma}\label{lem:special}
If $\{\gamma_1,\gamma_2\} \subseteq \Theta^{(r)}$ with $\gamma_1\ne\gamma_2$ and $r \ge \left \lfloor{m_\ell/2 }\right \rfloor+1$ (floor function), then $\{\gamma_1,\gamma_2\}$ is $A_1^2$. 
The assertion is true, in particular, if $r=m_{\ell-1}$.
\end{lemma}
\begin{proof}
We have $\gamma_1+\gamma_2 \notin \Phi$ and $\gamma_1-\gamma_2 \notin \sum_{\alpha \in \Delta}\mathbb Z_{\ge 0}\alpha$. 
Use the classification of irreducible root systems of rank $2$.
\end{proof}

\begin{theorem}\label{thm:crucial} 
If $\gamma_1,\gamma_2\in\Theta^{(m_{\ell-1})}$ with $\gamma_1\ne\gamma_2$, then for each $i \in \{1,2\}$
$$|\A^{H_{\gamma_i}}| - |\A^{H_{\gamma_1} \cap H_{\gamma_2}}|= m_{\ell-1}.$$
 \end{theorem} 
\begin{proof}
The formula holds true trivially for $\ell=3$, and we use the classification of irreducible root systems for $\ell=4$. 
Assume that $\ell\ge 5$. 
Then by Corollary \ref{cor:exactly2}, there are exactly two roots of height $m_{\ell-1}$. 
The rest follows from  \cite[Theorem 4.3]{MR20}.
 \end{proof}

The second key ingredient is Proposition \ref{prop:existence}. In fact, the $A_1^2$ sets described in Theorem \ref{thm:crucial} are not enough to generate all possible $W$-orbits of the $A_1^2$ sets (cf. Remark \ref{rem:numerical}). 
Notice that the focus of Theorem \ref{thm:half2} is non-RO $A_1^2$ sets.
Although, Theorem \ref{thm:crucial} alone is not enough to prove Theorem \ref{thm:half2}, it guarantees that the problem is solved if the involving non-RO $A_1^2$ set lies in the same $W$-orbit with a pair of roots of height $m_{\ell-1}$. 
\begin{proposition}\label{prop:existence}
If there exists an $A_1^2$ set that is not RO, then the set 
$\scS:  = \{ \{\gamma_1,\gamma_2\}\subseteq\Theta^{(m_{\ell-1})} \mid  \mbox{at least one of $\gamma_1,\gamma_2$ is a long root}\}$ contains a non-RO set.
\end{proposition}
 \begin{proof} 
Note that  $\scS \ne \emptyset$ since $\Theta^{(m_{\ell-1})}$ always contains a long root (Proposition \ref{prop:long}).
The case $\ell = 3$ is checked directly by the classification.
Assume that $\ell \ge 4$ and suppose to the contrary that every element in $\scS$ is RO.
We can take $\{\gamma_1,\gamma_2\} \in \scS$ and assume that $\gamma_1$ is long.
By Theorems \ref{thm:crucial} and \ref{thm:half1}, we have $m_{\ell-1}={\rm h}/2$. 
By Lemma \ref{lem:repeat}, $\ell =4$. 
By Corollary \ref{cor:existence-D4}, $\Phi=D_4$, and all $A_1^2$ sets must be RO. 
This contradicts the Proposition's assumption.
\end{proof}

Of course, if $\ell \ge 5$, then Corollary \ref{cor:exactly2} implies that the set $\scS$ contains only one element. However, the present statement is enough for us.

The third (and the final) key ingredient  is Corollary \ref{cor:same-component}. 
From the previous discussions, we will be in the situation that  there exist two non-RO $A_1^2$ sets which may form different orbits. 
 So we want to find a relation between them. 
\begin{corollary}\label{cor:same-component}
Assume that $\ell \ge 4$. 
If there are two $A_1^2$ sets $\{\theta, \lambda_1\}$,  $\{\theta, \lambda_2\}$ that are both non-RO, then both $\lambda_1$ and $\lambda_2$ lie in the unique irreducible component $\Omega$ of $\theta^\perp$ with $\mathrm{rank}(\Omega)\ge 2$.
\end{corollary}
\begin{proof}
The statement is trivial if $\theta^\perp$ is irreducible. 
If $\theta^\perp$ is reducible, then the statement follows from Propositions \ref{prop:non-A1} and \ref{prop:A1-component}.
\end{proof}

We need one more simple lemma.
\begin{lemma}\label{lem:orbit} 
 If $\{\beta_1, \beta_2\} \subseteq \Phi^+$ contains a long root and $(\beta_1, \beta_2)=0$, then $\{\beta_1, \beta_2\}$ lies in the same $W$-orbit with $\{\theta, \mu\}$ for some $\mu \in\Delta$. 

 \end{lemma}
 \begin{proof} 
 This is well-known.
 There is an irreducible component $\Psi\subseteq \theta^\perp$ such that $\{\beta_1, \beta_2\}$ lies in the same $W$-orbit with $\{\theta, \gamma\}$ for some $\gamma \in\Psi$. 
 There exist $\mu \in \Delta\cap\Psi$ and $w \in W(\Psi)$ such that $\gamma = w(\mu)$ and this $w$ fixes $\theta$, i.e., $ \theta = w( \theta)$. 
 Thus $\{\beta_1, \beta_2\}$ lies in the same $W$-orbit with $\{\theta,\mu\}$. 
\end{proof}

Now we are ready to prove Theorem \ref{thm:half2}.

\begin{proof}[Proof of Theorem \ref{thm:half2}]
It suffices to prove the theorem under the condition that the $A_1^2$ set  $\{\beta_1, \beta_2\}$ contains a long root. 
Otherwise, we would consider the dual root system $\Phi^\vee$ where short roots become long roots.
By Lemma \ref{lem:orbit}, $\{\beta_1, \beta_2\}$ lies in the same $W$-orbit with $\{\theta, \lambda_1\}$ for some $\lambda_1 \in \Delta$. 
We may also assume that $\ell \ge 4$ since the case $\ell=3$ is done in Remark \ref{rem:K0}(ii). 
We note that by Remark \ref{rem:K0}(i), proving Theorem \ref{thm:half2} is equivalent to showing that $\K_\theta(\lambda_1)=m_{\ell-1}-{\rm h}/2$ (notation in Proposition \ref{prop:3sums}).

By Proposition \ref{prop:existence}, we can find a non-RO set $\{\gamma_1,\gamma_2\}\subseteq\Theta^{(m_{\ell-1})}$ where $\gamma_1$ is a long root. 
Again by Lemma \ref{lem:orbit},  $\{\gamma_1,\gamma_2\}$ lies in the same $W$-orbit with $\{\theta, \lambda_2\}$ for some $\lambda_2 \in \Delta$. 
By Proposition \ref{prop:characterize}(i), $\{\theta, \lambda_1\}$ and $\{\theta, \lambda_2\}$ are $A_1^2$ and  non-RO.
Corollary \ref{cor:same-component} implies that $\lambda_1$ and $\lambda_2$ lie in the unique irreducible component $\Omega$ of $\theta^\perp$ with $\mathrm{rank}(\Omega)\ge 2$. 
We already know from Theorem \ref{thm:crucial} that Theorem \ref{thm:half2} is automatically proved for $\{\theta, \lambda_2\}$, i.e., $\K_\theta(\lambda_2)=m_{\ell-1}-{\rm h}/2$. 
So we want to prove that 
\begin{equation}\label{eq:final}
\K_\theta(\lambda_1)=\K_\theta(\lambda_2). 
\end{equation}

If $\|\lambda_1\|=\|\lambda_2\|$, then  $\{\theta, \lambda_1\}$ and $\{\theta, \lambda_2\}$ lie in the same $W$-orbit.
So Formula \eqref{eq:final} follows. 
Now consider $\|\lambda_1\|\ne\|\lambda_2\|$. 
Note that at most two root lengths occur in $\Omega$, they are $\|\lambda_1\|$ and $\|\lambda_2\|$.
Then the fact that $\{\theta, \lambda_1\}$ and $\{\theta, \lambda_2\}$ are both $A_1^2$ implies that $\{\theta, \beta\}$ is $A_1^2$ for all $\beta \in \Omega$. 
For $\beta \in \Omega$, with the notations in Definition \ref{def:N} applied to the $A_1^2$ set $\{\theta, \beta\}$, we have
$$\M_\theta(\beta) =
 \left\{ 
\Lambda \subseteq \Omega\middle|
\begin{array}{c}
       \mbox{$\Lambda$ is an irreducible subsystem of rank $2$, $\beta \in  \Lambda$},  \\
\mbox{$\Phi\cap \mbox{span}(\{\theta\} \cup \Lambda)$ is a reducible subsystem of rank $3$}.
    \end{array}
\right\}.$$
Indeed, if $\Lambda \subseteq \theta^\perp$ and $\Lambda$ is irreducible, then $\Lambda \subseteq \Omega$. 
We further make the following definition
$$\M'_\theta(\beta) :=
 \left\{ 
\Lambda \subseteq \Omega\middle|
\begin{array}{c}
       \mbox{$\Lambda$ is an irreducible subsystem of rank $2$, $\beta \in  \Lambda$},  \\
\mbox{$\Phi\cap \mbox{span}(\{\theta\} \cup \Lambda)$ is an irreducible subsystem of rank $3$}.
    \end{array}
\right\}.$$
Using Proposition \ref{prop:3sums} and Theorem \ref{exponents}(vi), we compute
\begin{align*}
2\K_\theta(\beta) & =  2\sum_{ \Lambda\in\M_\theta(\beta)}\sum_{ \delta \in \Lambda^+\setminus\{\beta\}}\left( \widehat{\beta}, \widehat\delta\right)^2 \\
  & =  {\rm h}(\Omega)- 2 - 2\sum_{ \Lambda\in\M'_\theta(\beta)}\sum_{ \delta \in \Lambda^+\setminus\{\beta\}}\left( \widehat{\beta}, \widehat\delta\right)^2.
\end{align*}
We claim that $\M'_\theta(\beta)= \emptyset$ for every $\beta \in \Omega$.
Suppose not and let $\Lambda \in \M'_\theta(\beta)$. 
Note that $\Gamma:=\Phi\cap \mbox{span}(\{\theta\} \cup \Lambda)$ admits $\theta$ as the highest root in its positive system $\Phi^+\cap \mbox{span}(\{\theta\} \cup \Lambda)$.
By a direct check on all rank-$3$ irreducible root systems and using the fact that $\{\theta, \beta\}$ is $A_1^2$ for all $\beta \in \Lambda \subseteq \Gamma \cap \Omega$, we obtain a contradiction. 
Thus the claim is proved and we have $\M'_\theta(\lambda_1)=\M'_\theta(\lambda_2)= \emptyset$.
By the computation above, $\K_\theta(\lambda_1)=\K_\theta(\lambda_2)={\rm h}(\Omega)/2- 1$. 
This completes the proof. 
\end{proof}

We close this section (\S\ref{sec:main}) by giving two corollaries. 
\begin{corollary}[Local-global inequalities]\label{cor:Local-global} 
Assume that a set $\{ \beta_1, \beta_2\}\subseteq \Phi$ is $A_1^2$. 
Recall the notation $\N_0=\N_0(\{ \beta_1,\beta_2\})$ in Definition \ref{def:N}.
For each $\Psi\in \N_0$, 
denote by $c_{\max}(\Psi)$ the largest coefficient of the highest root of the subsystem $\Psi$. Then
\begin{enumerate}
\item[(a)] $\sum_{\Psi \in \N_0}\left( m_2(\Psi)-1 \right) \ge m_2-1,$
\item[(b)]  $\sum_{\Psi \in \N_0}c_{\max}(\Psi)\ge c_{\max}$. 
\end{enumerate}
The equality in (a) (resp., (b)) occurs if and only if either (i) $\ell \le 4$, or (ii) $\ell \ge 5$ and $\{ \beta_1,  \beta_2\}$ is not RO.
\end{corollary}
\begin{proof}
Since $\ell \ge 3$, $m_2 = c_{\max}+1$ by Theorem \ref{thm:iso} and Corollary \ref{cor:criterion}. 
Thus (a) and (b) are essentially equivalent.
The left-hand sides of these inequalities are equal to $\K_0$ by Remark \ref{rem:K0}.
By Theorem \ref{thm:card} and Remark \ref{rem:K0},
$$\K_0=
\begin{cases}
{\rm h}/2-1 \mbox{ if $\{ \beta_1, \beta_2\}$ is RO}, \\
m_2 -1 \mbox{ if $\{ \beta_1, \beta_2\}$ is not RO}.
\end{cases}
$$
Thus, the inequalities follow. 
If $\ell =3$, the equalities always occur since ${\rm h}/2=m_2$.
If $\ell = 4$, we need only care about the case $\{ \beta_1, \beta_2\}$ is both $A_1^2$ and RO. 
This condition forces $\Phi=D_4$ by Corollary \ref{cor:existence-D4}. 
Again, we have ${\rm h}/2=m_2$. 
So the equalities alway occur if $\ell \le 4$.
If $\ell \ge 5$, by Lemma \ref{lem:repeat}, ${\rm h}/2>m_2$ . 
Thus the equalities occur if $\{ \beta_1,  \beta_2\}$ is not RO.
\end{proof}

\begin{corollary}\label{cor:Omega-cx} 
Let $\Omega$ be an irreducible component of $\theta^\perp$.  
If $\Phi$ is simply-laced, then the Coxeter number of $\Omega$ is given by
$${\rm h}(\Omega)=
\begin{cases}
2 \mbox{ if } \mathrm{rank}(\Omega)=1, \\
{\rm h} -2m_2+2 \mbox{ if } \mathrm{rank}(\Omega)\ge 2.
\end{cases}
$$
 \end{corollary}
 \begin{proof}
We need only give a proof for the second line.
Note that by Lemma \ref{lem:cases}, there exists at most one irreducible component $\Omega$ of $\theta^\perp$ satisfying $\mathrm{rank}(\Omega)\ge 2$. 
For every $\beta \in \Omega$, $\{\theta, \beta\}$ is $A_1^2$ since $\Phi$ is simply-laced. 
Moreover, $\{\theta, \beta\}$ is not RO by the reason of rank.
With the notations in Proof of Theorem \ref{thm:half2}, $\M'_\theta(\beta)=\emptyset$. 
It completes the proof. 
\end{proof}

\section{Proof of  Theorem \ref{thm:basis-main}}
\label{sec:Weyl-freeness part}
Theorem \ref{thm:basis-main} follows from Theorem \ref{thm:basis-non-RO} and Example \ref{ex:basis-RO} below.
\begin{theorem}\label{thm:basis-non-RO}
Assume that $X=H_1 \cap X_2$ is  $A_1^2$  but not RO. 
Let $\{\varphi_{1}, \ldots, \varphi_\ell\}$ be a basis for $D(\A)$ with $\deg {\varphi}_j=m_j$ ($1 \le j \le \ell$). 
Then $\{ {\varphi}^X_1, \ldots, {\varphi}^X_{\ell-2}\}$ is a basis for $D(\A^X)$.

 \end{theorem}
\begin{proof}
The statement is checked by a case-by-case method when $\ell \le 4$. 
If $\ell \ge 5$, then $m_{\ell-1}$ appears exactly once in $\exp(\A)$ (Corollary \ref{cor:exactly2}). 
By Remark \ref{rem:alternative}, $\A^{H_1}$ is free and $\{ {\varphi}^{H_1}_1, \ldots, {\varphi}^{H_1}_{\ell-1}\}$ is a basis for $D(\A^{H_1})$. 
By Theorem \ref{thm:card}, $|\A^{H_1}| - |\A^X|=m_{\ell-1}$.
By Proof of Theorem \ref{thm:main}, $\A^{H_1} \setminus \{X\}$ is also free. 
 Theorem \ref{thm:basis-derived} completes the proof.
\end{proof}

\begin{example}\label{ex:basis-RO}
Assume that $X\in L(\A)$ is both $A_1^2$ and RO. 
By Remark \ref{rem:numerical}, $\Phi=D_\ell$ with $\ell \ge 3$ ($D_3=A_3$). 
Suppose that 
$$Q=\prod_{1 \le i < j\le \ell}(x_i-x_j) \prod_{1 \le i < j\le \ell}(x_i+x_j),$$
where  $\{x_1,\ldots, x_\ell\}$ is an orthonormal basis for $V^*$. Let $H_1 = \ker(x_1+x_2)$, $H_2 = \ker(x_1-x_2)$, and $X=H_1 \cap X_2$. 
Then again by Remark \ref{rem:numerical}, $X$ is $A_1^2$ and RO. Define
\begin{align*}
\tau_i & := \sum_{k=1}^\ell x_k^{2i-1}\partial_k\,\, (1 \le i \le \ell-1),\\
\eta & := \sum_{k=1}^\ell \frac{x_1\ldots x_\ell}{x_k}\partial_k.
\end{align*}
Then it is known that $\tau_1,\ldots, \tau_{\ell-1},\eta$ form a basis for $D(\A)$. Let
$$\varphi:=\left(\prod_{k=3}^\ell(x_1^2-x_k^2)\right)\partial_1+\left(\prod_{k=3}^\ell(x_2^2-x_k^2)\right)\partial_2.$$
Then it is not hard to verify that $\varphi \in D(\A \setminus\{H_1\})$ and thus $(x_1+x_2)\varphi \in D(\A)$. 
By Saito's criterion, we may show that $\tau_1,\ldots, \tau_{\ell-2},\eta, (x_1+x_2)\varphi$ also form a basis for $D(\A)$, and $\tau_1,\ldots, \tau_{\ell-2},\eta, \varphi$ form a basis for $D(\A \setminus\{H_1\})$. 
Therefore,  $\{ \tau_1^{H_1}, \ldots, \tau_{\ell-2}^{H_1}, \eta^{H_1}\}$ is a basis for $D(\A^{H_1})$. 
This basis may have two elements having the same degree, for example, $\deg \eta = \deg \tau_{\ell/2}=\ell-1$ when $\ell$ is an even number.
However, it is easy to check that  for every case $\eta^X=0$. 
Hence, $\{ \tau_1^X, \ldots, \tau_{\ell-2}^X\}$ is a basis for $D(\A^X)$ and $\exp(\A^X)=\{1,3, 5,\ldots, 2\ell-5\}$ as predicted in Theorem \ref{thm:basis-derived}. 
This is also consistent with the fact that the $\A^X$ above is exactly the Weyl arrangement of type $B_{\ell-2}$.

 \end{example}

\section{Appendix: a bijective proof of Theorem \ref{thm:crucial}}\label{sec:app}

 In this section, we give an alternative and bijective proof for Theorem \ref{thm:crucial}. 
In comparison with the proof used in \S \ref{subsec:halfway2}, we do not use here the classification of rank-$4$ irreducible root systems.

\begin{definition}\label{def:local}
Let $\{\gamma_1,\gamma_2\}$ be an $A_1^2$ set. 
Let $\Psi \in\N_0=\N_0(\{ \gamma_1,\gamma_2\})$ (see notation in Definition \ref{def:N}).
For any $\alpha= \sum_{\mu \in \Delta(\Psi)} c_{\mu } {\mu} \in \Psi^+=\Phi^+ \cap \Psi$ ($\Delta(\Psi)$ is the base of $\Psi$ associated to $\Psi^+$), its \emph{local height} in $\Psi$ is defined by ${\rm ht}_{\Psi}(\alpha):=\sum_{\mu \in \Delta(\Psi)} c_{\mu }$.
\end{definition}
\begin{lemma}\label{lem:local height} 
With the notations and assumptions in Lemma \ref{lem:special}, Definition \ref{def:local} and Remark \ref{rem:K0}(ii), we have ${\rm ht}_{\Psi}(\gamma_i)=m_2(\Psi)$  for $i\in\{1,2\}$.
\end{lemma}
\begin{proof}
This follows from a direct verification on all rank-$3$ irreducible root systems. 
We will check only a (most) non-obvious case when $\Psi=  C_3$, and $\gamma_1=\mu_1+\mu_2, \gamma_2=2\mu_2+\mu_3$ (see Table \ref{tab:C3}).
Set $\alpha:=2\mu_1+2\mu_2+\mu_3$. 
Since $\alpha=2\gamma_1+\mu_3 \in \Psi^+ \subseteq \Phi^+$, we have ${\rm ht}(\alpha)>2{\rm ht}(\gamma_1)>m_\ell$, which is a contradiction.
\end{proof}

 \begin{table}[!ht]
\begin{center}
\begin{tabular}{cccc} 
\mbox{Height}  &  & & \\ 
5 & $2\mu_1+2\mu_2+\mu_3$ & & \\ 
4 & $\mu_1+2\mu_2+\mu_3$ & & \\ 
3 & $\mu_1+\mu_2+\mu_3$ &  $2\mu_2+\mu_3$ & \\ 
2 & $\mu_1+\mu_2$ &  $\mu_2+\mu_3$ & \\ 
1 & $\mu_1$ & $\mu_2$ &  $\mu_3$ 
\end{tabular}
\vskip 1em
\caption{$\Psi^+$ when $\Psi=  C_3$.}
\label{tab:C3}
\end{center}
\end{table}

\begin{proof}[A bijective proof of Theorem \ref{thm:crucial}]
By Remark \ref{rem:K0}, Theorem \ref{thm:crucial} is proved once we prove 
\begin{equation}\label{eq:once}
\sum_{\Psi\in\N_0}\left( m_2(\Psi)-1 \right)=m_2-1. 
\end{equation}
Recall the notation $\U=\{\theta_j\in \Phi^+ \mid {\rm ht}(\theta_j) >m_{\ell-1}\}$, and $|\U|=m_2-1$.
For each $\Psi\in\N_0$, set $\U_\Psi:=\{\mu \in \Psi^+ \mid  {\rm ht}_{\Psi}(\mu)>{\rm ht}_{\Psi}(\gamma_i)\}$ (by Lemma \ref{lem:local height}, this definition does not depend on the index $i$).
Moreover, $|\U_\Psi|=m_2(\Psi)-1$. 
Since $\U_{\Psi} \cap \U_{\Psi'} = \emptyset$ for $\Psi \ne \Psi'$, we have
$$\left|\bigcup_{\Psi\in\N_0}\U_\Psi \right| = \sum_{\Psi\in\N_0}|\U_\Psi|=\sum_{\Psi\in\N_0}\left( m_2(\Psi)-1 \right).$$
Equality \eqref{eq:once} will be proved once we prove $\U=\bigcup_{\Psi\in\N_0}\U_\Psi$. 
For any $\mu \in \U_\Psi$,
$$\mu - \gamma_i =  \sum_{\mu \in \Delta(\Psi)}\mathbb Z_{\ge 0}\mu \subseteq   \sum_{\alpha \in \Delta}\mathbb Z_{\ge 0}\alpha \quad (i\in\{1,2\}).$$
Thus, ${\rm ht}(\mu) >{\rm ht}(\gamma_i)=m_{\ell-1}$ hence $\mu \in \U$.
Therefore, $\U \supseteq \bigcup_{\Psi\in\N_0}\U_\Psi$.
To prove the inclusion, it suffices to prove that for every $\theta_j \in \U$, the subsystem $\Gamma:=\Phi\cap \mbox{span}\{\theta_j ,\gamma_1,\gamma_2\}$ is an element of $\N_0$.
Obviously, $\Gamma$ is of rank $3$ since it contains the $A_1^2$ set $\{\gamma_1,\gamma_2\}$.
The irreducibility of $\Gamma$ follows from Proposition \ref{prop:irr}.
\end{proof}

\noindent
\textbf{Acknowledgements.} 
The current paper is an improvement of the third author's Master's thesis, written under the supervision of the second author at Hokkaido University in 2017. 
At that time, the first main result (Theorem \ref{thm:combine}) was only proved by a case-by-case method.
The second author was supported by 
JSPS Grants-in-Aid for basic research (A) No. 24244001.
The third author was partially supported by the scholarship program of Japanese Ministry of Education, Culture, Sports, Science, and Technology 
(MEXT) No. 142506, and is currently supported by JSPS Research Fellowship for Young Scientists No. 19J12024.

\bibliographystyle{alpha} 
\bibliography{references}

\end{document}